\documentclass[11pt, reqno]{amsart}
\usepackage{amsmath, amssymb, braket,bm}%amsCD
\usepackage{color}
\usepackage{graphicx}%[dvipdfmx]
\usepackage{tikz}
\usepackage[all]{xy}
\usepackage{url}
\newtheorem{theorem}{Theorem}[section]
\newtheorem{lemma}[theorem]{Lemma}
\newtheorem{proposition}[theorem]{Proposition}
\newtheorem{corollary}[theorem]{Corollary}

\theoremstyle{definition}
\newtheorem{definition}[theorem]{Definition}

\theoremstyle{remark}
\newtheorem{remark}[theorem]{Remark}
\numberwithin{equation}{section}

\newcommand{\Z}{\mathbb{Z}}
\newcommand{\Q}{\mathbb{Q}}

\newcommand{\C}{\mathcal{C}}

\newcommand{\I}{\mathcal{I}}
\newcommand{\J}{\mathcal{J}}

\newcommand{\M}{\mathcal{M}}

\newcommand{\oo}{\mathit{oo}}

\newcommand{\rr}{\mathfrak{r}}
\newcommand{\A}{\mathcal{A}}
\renewcommand{\H}{\mathcal{H}}

\newcommand{\Mag}{\operatorname{Mag}}
\newcommand{\Der}{\operatorname{Der}}

\renewcommand{\Im}{\operatorname{Im}}
\newcommand{\id}{\mathrm{id}}
\newcommand{\Ker}{\operatorname{Ker}}

\newcommand{\Aut}{\operatorname{Aut}}
\newcommand{\IAut}{\operatorname{IAut}}
\newcommand{\incl}{\mathrm{incl}}

\newcommand{\Int}{\operatorname{Int}}
\newcommand{\Diff}{\mathrm{Diff}}
\newcommand{\Hom}{\mathrm{Hom}}
\newcommand{\sign}{\operatorname{sign}}

\newcommand{\Eul}{\mathrm{Eul}}

\newcommand{\pr}{\mathrm{pr}}

\newcommand{\GL}{\mathrm{GL}}

\newcommand{\ab}{\mathrm{ab}}

\newcommand{\ev}{\mathrm{ev}}
\newcommand{\plim}{\varprojlim}

\newcommand{\Lie}{\mathrm{Lie}}
\newcommand{\comm}{\mathrm{comm}}
\newcommand{\ldet}{\ell\!\operatorname{det}}
\newcommand{\Tr}{\operatorname{Tr}}
\newcommand{\tr}{\operatorname{tr}}
\newcommand{\cl}{\operatorname{cl}}
\newcommand{\col}{\mathrm{color}}

\title[A non-commutative Reidemeister-Turaev torsion]{A non-commutative Reidemeister-Turaev torsion of homology cylinders
\footnote{
\lowercase{\uppercase{F}irst electronically published in \uppercase{T}rans.\ \uppercase{A}mer.\ \uppercase{M}ath.\ \uppercase{S}oc.\ on \uppercase{A}pril 19, 2023, \uppercase{DOI}: \url{https://doi.org/10.1090/tran/8925} (to appear in print), published by the \uppercase{A}merican \uppercase{M}athematical \uppercase{S}ociety.
}}}

\author{Yuta Nozaki}
\address{
Faculty of Environment and Information Sciences, Yokohama National University \\
79-7 Tokiwadai, Hodogaya-ku, Yokohama, 240-8501 \\
Japan\vspace{-0.6em}}
\address{
SKCM$^2$, Hiroshima University \\
1-3-2 Kagamiyama, Higashi-Hiroshima City, Hiroshima, 739-8511 \\
Japan}
\email{nozaki-yuta-vn@ynu.ac.jp}

\author{Masatoshi Sato}
\address{
Department of Mathematics \\
Tokyo Denki University \\
5 Senjuasahi-cho, Adachi-ku, Tokyo 120-8551 \\
Japan}
\email{msato@mail.dendai.ac.jp}

\author{Masaaki Suzuki}
\address{Department of Frontier Media Science, Meiji University \\
4-21-1 Nakano, Nakano-ku, Tokyo, 164-8525 \\
Japan}
\email{mackysuzuki@meiji.ac.jp}

\makeatletter
\@namedef{subjclassname@2020}{\textup{2020} Mathematics Subject Classification}
\makeatother

\subjclass[2020]{Primary 57K16, 57Q10, 19B28, Secondary 57K31, 57K20, 17B01}
\keywords{Homology cylinder, surgery map, clasper, LMO functor, Reidemeister torsion, Enomoto-Satoh trace, $K_1$-group}

\begin{document}
\maketitle

\begin{abstract}
We compute the Reidemeister-Turaev torsion of homology cylinders which takes values in the $K_1$-group of the $I$-adic completion of the group ring $\mathbb{Q}\pi_1\Sigma_{g,1}$,
and prove that its reduction to $\widehat{\mathbb{Q}\pi_1\Sigma_{g,1}}/\hat{I}^{d+1}$ is a finite-type invariant of degree $d$.
We also show that the $1$-loop part of the LMO homomorphism and the Enomoto-Satoh trace can be recovered from the leading term of our torsion.
\end{abstract}

\setcounter{tocdepth}{1}
\tableofcontents

%%%%%%%%%
\section{Introduction}\label{sec:Introduction}
The  graph-valued invariant of a knot in $S^3$, called the Kontsevich invariant, was extended to an invariant of links in arbitrary closed 3-manifolds called the Le-Murakami-Ohtsuki invariant in \cite{LMO98}.
Cheptea Habiro and Massuyeau~\cite{CHM08} extended it to the case of 3-manifolds with boundary and constructed the LMO functor, which is a functor from some category of cobordisms between surfaces to a category of graphs.

Let $\Sigma_{g,1}$ be a compact oriented surface of genus $g$ with one boundary component.
A homology cylinder over $\Sigma_{g,1}$ is a compact oriented 3-manifold which is homologically the same as the product $\Sigma_{g,1}\times [-1,1]$.
The set of homology cylinders $\I\C=\I\C_{g,1}$ forms a monoid,
and by restricting the LMO functor to $\I\C$,
it induces a monoid homomorphism from $\I\C$ to some Hopf algebra of Jacobi diagrams,
where a Jacobi diagram is a uni-trivalent graph with additional information.
We call the number of trivalent vertices in a Jacobi diagram the \emph{degree}.
Let us consider a Jacobi diagram without strut component whose univalent vertices are labeled by elements of $H$,
and denote by $\A(H)$ the degree completion of the $\Q$-vector space spanned by such Jacobi diagrams modulo relations, called the AS, IHX, and multi-linear relations.
The completion $\A(H)$ has a Hopf algebra structure,
and Habiro and Massuyeau~\cite[Section~4]{HaMa12} considered a monoid homomorphism
\[
Z=\kappa\circ \widetilde{Z}^Y\colon \I\C\to \A(H)
\]
called the LMO homomorphism,
where $\widetilde{Z}^Y$ is the ``$Y$-part'' of the LMO functor,
and $\kappa$ is some isomorphism between spaces of Jacobi diagrams in \cite[Section~4]{HaMa12}.
See also Habiro and Massuyeau~\cite{HaMa09}.
Jacobi diagrams have another grading given by the number of loops,
and Massuyeau \cite{Mas12} showed that the $0$-loop part of the LMO homomorphism $Z(M)$ for $M\in\I\C$ is written in terms of the total Johnson map.

In this paper, we give an interpretation of the 1-loop part of the LMO homomorphism by considering a non-commutative Reidemeister(-Turaev) torsion of homology cylinders which takes values in $K_1(\widehat{\Q\pi})$,
where $\pi=\pi_1\Sigma_{g,1}$
and $\widehat{\Q\pi}$ is the $I$-adic completion of the group ring $\Q\pi$.
More precisely,
we show that the torsion is a finite-type invariant of homology cylinders
and that the ``leading term'' of the $1$-loop part of the LMO homomorphism and the Enomoto-Satoh trace of homology cylinders are essentially equal to the ``leading term'' of the torsion.
We also give a surgery formula of the torsion for $k$-loop graph claspers when $k\ge1$.

Our work is motivated by the study of the loop expansion of the Kontsevich invariant and the LMO invariant.
For a knot $K$ in the 3-sphere $S^3$,
the 1-loop part of the logarithm $\log Z(K)$ of the Kontsevich invariant is known to be written in terms of the Alexander polynomial of $K$.
See~\cite[(4.11)]{Roz03} for the explicit formula.
This follows from the Melvin-Morton-Rozansky conjecture~\cite{MeMo95} proved by Bar-Natan and Garoufalidis~\cite{BaGa96}.
See Ito~\cite{Ito19} for another proof and Rozansky~\cite{Roz98} for some generalization.
The Alexander polynomial also plays an important role to describe higher loop parts.
For a knot $K$ in an integral homology 3-sphere $M$,
the logarithm of the LMO invariant decomposes into an infinite sum of $n$-loop part $Q_n(M,K)$ for $n\ge1$.
Rozansky's rationality conjecture~\cite{Roz03} proved by Kricker~\cite{Kri00} states that $Q_n(M,K)$ is described by some rational functions whose denominators are the Alexander polynomial for $n\ge1$.
See also Garoufalidis and Kricker~\cite{GaKr04} for the Kontsevich invariant of links.

The relation between the Alexander polynomial and the Reidemeister(-Turaev) torsion of 3-manifolds is given by Milnor and Turaev.
In \cite{Mil62},
Milnor considered the Alexander polynomial of link complements and interpreted it as a kind of Reidemeister torsion.
Turaev \cite{Tur86} generalized it to compact oriented 3-manifolds.
Turaev also introduced an Euler structure to construct a refinement of the Reidemeister torsion in \cite{Tur86} and \cite{Tur89},
which is now called the Reidemeister-Turaev torsion.
Let $H_\Z=H_1(\Sigma_{g,1};\Z)$ and denote by $Q(\Z H_\Z)$ the field of fractions of the group ring $\Z H_\Z$.
For homology cylinders, Friedl-Juh\'{a}sz-Rasmussen~\cite{FJR11} and Massuyeau-Meilhan~\cite{MaMe13} considered the (relative) Reidemeister-Turaev torsion which takes values in $K_1(Q(\Z H_\Z))$.
Massuyeau and Meilhan also constructed a canonical Euler structure $\xi_0$ for homology cylinders, which they call the preferred Euler structure,
and constructed a monoid homomorphism
\[
\alpha\colon \I\C\to K_1(Q(\Z H_\Z))
\]
defined by the torsion.
Our torsion $\tilde{\alpha}\colon \I\C\to K_1(\widehat{\Q\pi})$ is considered as a lift of their torsion,
and we extend some of their results to our torsion.

The paper is organized as follows.
In Section~\ref{section:pre},
we review the Reidemeister(-Turaev) torsion of pairs of CW-complexes and give precise definitions of homology cylinders and two filtrations on the homology cylinders called the $Y$-filtration and the Johnson filtration.
In Section~\ref{section:reidemeistertorsion},
we introduce a $K_1(\widehat{\Q\pi})$-valued torsion and construct a homomorphism 
\[
\tilde{\alpha}\colon \I\C\to K_1(\widehat{\Q\pi})
\]
using the preferred Euler structure $\xi_0$.
We also show that our torsion $\tilde{\alpha}$ is essentially a lift of $\alpha\colon \I\C\to K_1(Q(\Z H_\Z))$.
Let $H=H_1(\Sigma_{g,1};\Q)$.
In Section~\ref{sec:K1_of_Qpi},
we determine the group structure of $K_1(\widehat{\Q\pi})$ as follows:

\begin{theorem}
\label{thm:K1Qpihat}
The group $K_1(\widehat{\Q\pi})$ is isomorphic to $\Q^\times\oplus\prod_{n=1}^\infty (H^{\otimes n})_{\Z_n}$.
\end{theorem}

Here, the cyclic group $\Z_n$ acts on $H^{\otimes n}$ by the cyclic permutation.
Note that the isomorphism in Theorem~\ref{thm:K1Qpihat} is not canonical,
and it depends on the choice of a Magnus expansion.
In Section~\ref{section:finitetype},
we show that the $I$-adic reduction of $\tilde{\alpha}\colon \I\C\to K_1(\widehat{\Q\pi})$ is a finite-type invariant of homology cylinders.
More precisely, we prove:

\begin{theorem}\label{thm:Y-filtration}
Let $M\in \I\C$,
and let $V_1,\ldots, V_d$ denote disjoint handlebodies embedded in $M$.
Choose mapping classes $\varphi_i$ of the surface $\partial V_i$ which acts on $H_1(\partial V_i;\Z)$ trivially.
For a subset $J\subset\{1,2,\ldots,d\}$,
let us denote
\[
M_J=\left(M\setminus \bigsqcup_{i\in J} \Int V_i\right) \cup_{\{\varphi_i\}_{i\in J}} \bigsqcup_{i\in J}V_i.
\]
Then we have
\[
\prod_{J\subset\{1,2,\ldots, d\}}\tilde{\alpha}(M_J)^{(-1)^{|J|}} \in \Ker\left(K_1(\widehat{\Q\pi})\to K_1(\widehat{\Q\pi}/\hat{I}^d)\right).
\]
\end{theorem}

In Section~\ref{section:change-rt},
we show that $\tilde{\alpha}$ does not change under $k$-loop clasper surgeries for $k\ge2$ (Theorem~\ref{thm:k-loop}).
We also give a surgery formula (Theorem~\ref{thm:value-1-loop}) for $\tilde{\alpha}$ with respect to $1$-loop clasper surgeries.

Let $\{Y_d\I\C\}_{d\ge1}$ denote the $Y$-filtration on the monoid $\I\C$ of homology cylinders.
The precise definition is given in Section~\ref{sec:homologycylinders}.
Since the torsion $\tilde{\alpha}\colon \I\C\to K_1(\widehat{\Q\pi}/\hat{I}^{d+1})$ is a finite-type invariant of degree at most $d$,
it induces a homomorphism
\[
Y_d\I\C/Y_{d+1}\to \Ker\left(K_1(\widehat{\Q\pi}/\hat{I}^{d+1}) \to K_1(\widehat{\Q\pi}/\hat{I}^d)\right),
\]
where $Y_d\I\C/Y_{d+1}$ denotes the abelian group consisting of $Y_{d+1}$-equivalence classes.
Composing it with a canonical isomorphism
\[
\Ker\left(K_1(\widehat{\Q\pi}/\hat{I}^{d+1})\to K_1(\widehat{\Q\pi}/\hat{I}^d)\right) \cong (H^{\otimes d})_{\Z_d}
\]
given in Corollary~\ref{cor:gr}, 
we write it as
\[
\tilde{\alpha}_d\colon Y_d\I\C/Y_{d+1}\to (H^{\otimes d})_{\Z_d}.
\]
In Section~\ref{section:proof-of-main}, we prove the main theorem (Theorem~\ref{thm:main}) which relates $\tilde{\alpha}_d$ with the 1-loop part of the LMO homomorphism $Z\colon\I\C\to \A(H)$ and the Enomoto-Satoh trace.

Let us fix notation for the Johnson homomorphism and the Enomoto-Satoh trace.
We write $J_d\C\subset \C$ for the $d$th submonoid of the Johnson filtration on $\C$ (see Section~\ref{subsec:Johnson}) and $L_d(H_\Z)$ for the degree $d$ part of the free Lie ring generated by $H_\Z$.
Let $x\cdot y\in\Z$ denote the intersection number of $x,y\in H_\Z$,
and denote by
\[
\tau_d\colon J_d\C\to \Hom(H_\Z,L_{d+1}(H_\Z))\cong H_\Z\otimes L_{d+1}(H_\Z)
\]
the $d$th Johnson homomorphism,
where the right isomorphism is induced by Poincar\'e duality $H_\Z\cong H_\Z^*$ which sends $x\mapsto x\cdot \textendash$.
It induces a homomorphism on the graded quotient $Y_d\I\C/Y_{d+1}$.
The Enomoto-Satoh trace is defined by the composition map
\[
\Tr_d\colon H_\Z\otimes H_\Z^{\otimes(d+1)}\xrightarrow{C}H_\Z^{\otimes d} \twoheadrightarrow (H_\Z^{\otimes d})_{\Z_d},
\]
where the first homomorphism is the contraction
\[
C(x\otimes x_1\otimes x_2\otimes\cdots \otimes x_{d+1})=(x\cdot x_1)x_2\otimes \cdots \otimes x_{d+1},
\]
and the second one is the projection.
Let $\rr\colon (H^{\otimes d})_{\Z_d}\to (H^{\otimes d})_{\Z_d}$ be an involution defined by
\[
\rr(x_1\otimes x_2\otimes\cdots\otimes x_d)=(-1)^dx_d\otimes\cdots\otimes x_2\otimes x_1,
\]
and denote by $(H^{\otimes d})_{\Z_d}^{+}$ and $(H^{\otimes d})_{\Z_d}^{-}$ the $(+1)$- and $(-1)$-eigenspaces of $\rr$, respectively.
As shown by Conant~\cite[Theorem~4.2(1)]{Con15},
the image of the composition map
\[
\Tr_d\circ \tau_d\colon J_d\C\to (H_\Z^{\otimes d})_{\Z_d}
\]
lies in $(H^{\otimes d})_{\Z_d}^-$.
Here, we regard $L_{d+1}(H_\Z)$ as a subgroup of $H_\Z^{\otimes(d+1)}$ via $[x,y]=x\otimes y-y\otimes x$.

Let $\A_d^c(H)$ denote the submodule of $\A(H)$ generated by connected Jacobi diagrams of degree $d$.
The projection of the LMO homomorphism $Z\colon \I\C\to \A(H)$ to the degree $d$ part induces a homomorphism $Z_d\colon Y_d\I\C/Y_{d+1}\to \A^c_d(H)$,
and we denote by
\[
Z_{d,1}\colon Y_d\I\C/Y_{d+1}\to \A_{d,1}^c(H)
\]
its projection to the $1$-loop part.
Let $p_+\colon (H^{\otimes d})_{\Z_d}\to \A_{d,1}^c(H)$ be a homomorphism defined by
\[
p_+(x_1\otimes x_2\otimes\cdots\otimes x_d)
=O(x_1,x_2,\ldots,x_d),
\]
where
\[
O(x_1,x_2,\dots,x_d)=
\begin{tikzpicture}[scale=0.25, baseline={(0,0.1)}, densely dashed]
 \draw (0,0) circle [radius=2];
 \draw (90:2)--(90:4) node[anchor=south] {$x_1$};
 \draw (50:2)--(50:4) node[anchor=south west] {$x_2$};
 \draw (10:2)--(10:4) node[anchor=west] {$x_3$};
 \draw[loosely dotted, very thick] (-20:4) arc (-20:-50:4);
 \draw (130:2)--(130:4) node[anchor=south east] {$x_d$};
 \draw[loosely dotted, very thick] (160:4) arc (160:190:4);
\end{tikzpicture}.
\]
The homomorphism $p_+$ induces an isomorphism $(H^{\otimes d})_{\Z_d}^+\to \A_{d,1}^c(H)$ as we can deduce from \cite[Proposition~5.1]{NSS22GT}.
We also denote by $p_-\colon (H^{\otimes d})_{\Z_d} \to (H^{\otimes d})_{\Z_d}^-$ the projection defined by
\[
p_-(x_1\otimes x_2\otimes\cdots\otimes x_d)
=\frac{1}{2}(x_1\otimes x_2\otimes\cdots\otimes x_d-(-1)^dx_d\otimes \cdots \otimes x_2\otimes x_1).
\]
Our main theorem describes the leading term $\tilde{\alpha}_d$ as follows:

\begin{theorem}\label{thm:main}
For $d\ge 1$, the diagram
\[
\xymatrix{
&Y_d\I\C/Y_{d+1}\ar[d]^-{\tilde{\alpha}_d}\ar[ld]_-{-\frac{1}{2}\Tr_d\circ \tau_d}\ar[rd]^-{-2Z_{d,1}}&\\
(H^{\otimes d})_{\Z_d}^-&(H^{\otimes d})_{\Z_d}\ar[l]^-{p_-}\ar[r]_-{p_+}& \A_{d,1}^c(H)
}
\]
commutes.
\end{theorem}

In other words, the $(+1)$- and $(-1)$-eigenspaces of $\tilde{\alpha}_d$ are essentially the 1-loop part of the LMO homomorphism and the Enomoto-Satoh trace, respectively.

\subsection*{Acknowledgments}
The authors would like to thank
Tetsuya Ito, Nariya Kawazumi, Yusuke Kuno, and Takuya Sakasai for helpful comments.
In particular, Yusuke Kuno gave the authors an idea to use symplectic expansions in the proof of Lemma~\ref{lem:SympExp}.
This study was supported in part by JSPS KAKENHI Grant Numbers JP20K14317, JP18K03310, JP22K03298, JP19H01785, and JP20K03596.
The authors would like to thank the referee for his/her careful reading of the manuscript and for various comments.

%%%%%%%
\section{Preliminaries}
\label{section:pre}
In this section,
we recall the definitions of the Reidemeister torsion and its refinement called the Reidemeister-Turaev torsion of pairs of CW-complexes.
For details on the Reidemeister(-Turaev) torsion, see \cite[Section~3]{Mil66}, \cite[Chapters~I and II]{Tur01}, \cite[Chapters III and IV]{Coh73}, and \cite{Mas11}.
We also review homology cylinders and two filtrations on the monoid of homology cylinders called the $Y$-filtration and the Johnson filtration.
Homology cylinders are introduced by Goussarov~\cite{Gou99} and Habiro~\cite{Hab00C} in connection with finite-type invariants of 3-manifolds.
See also \cite{GaLe05}, \cite{HaMa09, HaMa12} and \cite{NSS22GT,NSS22JT} for related topics.

In this paper, all homology groups are with rational coefficients, if not specified.
Throughout this section, 
$R$ denotes a local ring or a commutative PID. 

\subsection{The torsion of a chain complex}
Here, we review the torsion of an acyclic chain complex over $R$.
Let $C$ be an acyclic chain complex of finitely generated free $R$-modules:
\[
C_*=(0\to C_m\xrightarrow{\partial_m} C_{m-1}\xrightarrow{\partial_{m-1}}\cdots \xrightarrow{\partial_2}C_1\xrightarrow{\partial_1} C_0\to 0).
\]
The submodule $B_i=\Im\partial_{i+1}$ is a free $R$-module of finite rank when $R$ is a PID.
Since $C_*$ is acyclic, this is also true when $R$ is a local ring as in \cite[Lemma~1.2]{Mil71}.
Since $C_i$ and $B_i$ are free $R$-modules for $0\le i\le m$,
we can choose and fix their ordered bases $c_i$ and $b_i$, respectively.
Taking a lift $\tilde{b}_{i-1}$ of $b_{i-1}$,
we obtain another ordered basis of $C_i$ which we denote by $b_i\tilde{b}_{i-1}$.
We compare it with $c_i$, and denote
\[
[b_i\tilde{b}_{i-1}/c_i]=(\text{matrix expressing $b_i\tilde{b}_{i-1}$ w.r.t. the basis $c_i$}).
\]

\begin{definition}
The torsion of $C_*$ based by $c=\{c_i\}_{i=0}^m$ is defined by
\[
\tau(C_*,c)=\prod_{i=0}^m[b_i\tilde{b}_{i-1}/c_i]^{(-1)^{i+1}}\in K_1(R).
\]
\end{definition}

\begin{remark}
The torsion $\tau(C_*,c)$ does not depend on the choices of bases $\{b_i\}_{i=0}^m$ and their lifts $\{\tilde{b}_{i-1}\}_{i=0}^m$, where $\tilde{b}_{-1}$ denotes the empty set.
\end{remark}

\begin{lemma}\label{lem:choiceofC}
Let $c_i$ and $c_i'$ be ordered bases of $C_i$.
Then, we have
\[
\tau(C_*,\{c'_i\})=\tau(C_*,\{c_i\})\prod_{i=0}^m[c_i/c'_i]^{(-1)^{i+1}},
\]
where $[c_i/c'_i]\in K_1(R)$ is the matrix expressing $c_i$ with respect to the basis $c'_i$.
\end{lemma}

\subsection{The torsion of a pair of CW-complexes}
\label{section:torsion}
Let $X$ be a connected finite CW-complex, and $Y$ a connected subcomplex.
Let $p\colon \tilde{X}\to X$ be the universal covering,
and choose points $*$ in $Y$ and $\tilde{*}$ in $p^{-1}(*)$.
For a based loop $\gamma\in\pi_1(X,*)$,
we denote by $t_\gamma$ the deck transformation which sends $\tilde{*}$ to $\tilde{\gamma}(1)$,
where $\tilde{\gamma}$ is the unique lift of $\gamma$ in $\tilde{X}$ starting at $\tilde{*}$.
For a homomorphism $\rho\colon\mathbb{Z}\pi_1(X,*)\to R$,
consider the chain complex
\[
C_*(X,Y;R)=C_*(\tilde{X}, p^{-1}(Y);\Z)\otimes_{\mathbb{Z}\pi_1(X,*)}R,
\]
where the right action of $\gamma\in \pi_1(X,*)$ on $C_*(\tilde{X}, p^{-1}(Y);\Z)$ is given by $c\cdot \gamma=t_\gamma^{-1}(c)$ and the left action on $R$ is given by $\gamma\cdot r=\rho(\gamma)r$.
Note that, in \cite{Mil66} and \cite{Tur01}, chain complexes of the form $R\otimes_{\mathbb{Z}\pi_1(X,*)}C_*(\tilde{X}, p^{-1}(Y);\Z)$ are considered, 
and our torsion depends on the convention used in its definition.

The module $C_i(X,Y;R)$ is a finitely generated free right $R$-module.
Let $E$ denote the set of cells in $X\setminus Y$.
For each cell $e\in E$, we choose a lift $\tilde{e}$ to $\tilde{X}\setminus p^{-1}(Y)$.
We denote by $\tilde{E}$ the set of the lifted cells.
We also put a total ordering on $E$,
and choose an orientation for each cell $e\in E$.
This double choice is denoted by $\oo$.
The choice of $\tilde{E}$ combined to $\oo$ induces a basis $\tilde{E}_{\oo}$ of $C_*(\tilde{X}, p^{-1}(Y);\Z)$,
which defines itself a basis $\tilde{E}_{\oo}\otimes 1$ of $C_*(X,Y;R)$.

\begin{definition}
The Reidemeister torsion with $\rho$-twisted coefficients of the pair $(X,Y)$ of CW-complexes is
\[
\tau(C_*(X,Y;R),\tilde{E}_{\oo}\otimes 1)\in K_1(R)/{\pm\rho(H_1(X;\Z))},
\]
where $\rho\colon H_1(X;\Z)\to K_1(R)$ is the induced homomorphism by $\rho\colon \pi_1X\to R^\times$.
\end{definition}

\subsection{The Reidemeister-Turaev torsion}\label{section:RT-torsion}
Turaev gives refinements of commutative Reidemeister torsions in \cite{Tur89}.
Here, we review the construction following \cite{Mas11},
noting that the construction is also valid for non-commutative Reidemeister torsions.

In order to lift $\tau(C_*(X,Y;R),\tilde{E}_{\oo}\otimes 1)$ to an element in $K_1(R)$,
there are two ambiguities which come from the choice of a basis of $C_*(\tilde{X}, p^{-1}(Y);\Z)$; the sign indeterminacy and the $\rho(H_1(X;\Z))$-indeterminacy.
The sign indeterminacy comes from the choice of $oo$,
and $\rho(H_1(X;\Z))$-indeterminacy comes from the choice of $\tilde{E}$.
When the chain complex $C_*(X,Y;\Q)$ is also acyclic,
we can eliminate the sign indeterminacy.
\begin{lemma}
The element
\[
\tau^\rho(X,Y,\tilde{E})=(\sign\tau(C_*(X,Y;\Q), \oo))\tau(C_*^\rho(X,Y;R),\tilde{E}_{\oo}\otimes 1)\in K_1(R)
\]
does not depend on the choice of $\oo$.
\end{lemma}
We also denote
\[
\tau^\rho(X,Y)=(\sign\tau(C_*(X,Y;\Q), \oo))\tau(C_*^\rho(X,Y;R),\tilde{E}_{\oo}\otimes 1)\in K_1(R)/\rho(H_1(X;\Z)),
\]
which does not depend on the choices of $\tilde{E}$ and $\oo$.
To eliminate the $\rho(H_1(X;\Z))$-indeterminacy,
we review the notion of fundamental family of cells.
A family $\tilde{E}$ of cells of $\tilde{X}\setminus p^{-1}(Y)$ is said to be fundamental if each cell $e$ of $X\setminus Y$ has a unique lift $\tilde{e}$ in $\tilde{E}$.
Two fundamental families of cells $\tilde{E}$ and $\tilde{E}'$ are equivalent when the alternating sum
\[
\sum_{e\in E}(-1)^{\dim(e)}\left[\overrightarrow{\tilde{e}\tilde{e}'}\right]\in H_1(X;\Z)
\]
vanishes, where $\overrightarrow{\tilde{e}\tilde{e}'}\in \pi_1X$ is the loop such that the covering transformation $t_{\overrightarrow{\tilde{e}\tilde{e}'}}$ maps $\tilde{e}$ to $\tilde{e}'$.
For an equivalence class $[\tilde{E}]$ and $[\gamma]\in H_1(X;\Z)$ represented by $\gamma\in\pi_1X$,
let us denote by $[\tilde{E}]+[\gamma]$ another equivalence class obtained by replacing one cell $\tilde{e}$ of $\tilde{E}$ with $t_\gamma^{\epsilon}(\tilde{e})$, where $\epsilon=(-1)^{\dim (\tilde{e})+1}$.
Note that the sign $\epsilon$ is different from \cite[Section~3.2]{Tur89} and \cite[Section~7.2]{Mas11}. 
Denote by $\widehat{\Eul}(X,Y)$ the set of equivalence classes of fundamental families of cells,
which is an $H_1(X;\Z)$-affine space under the above action
and is identified with the set of combinatorial Euler structures in \cite{Tur01} and \cite[Section~7]{Mas11}.
For an equivalence class $\xi=[\tilde{E}]\in \widehat{\Eul}(X,Y)$,
let us denote
\[
\tau^\rho(X,Y, \xi)=\tau^\rho(X,Y,\tilde{E})\in K_1(R),
\]
which does not depend on the choice of a representative $\tilde{E}$.
It is called the \emph{Reidemeister-Turaev torsion} of the pair $(X,Y)$ equipped with $\xi$.

Let us denote the composition map
\[
\pi_1X\xrightarrow{\rho} R^\times\cong\GL(1,R)\to K_1(R)
\]
by the same symbol as $\rho\colon \Z\pi_1X\to R$.
Note that it factors through $H_1(X;\Z)$ since $K_1(R)$ is abelian.
If we change an equivalence class $\xi$ of fundamental families of cells,
the Reidemeister-Turaev torsion changes as follows.
\begin{lemma}\label{lem:euler structure}
For $\gamma\in \pi$,
\[
\tau^{\rho}(X,Y,\xi+[\gamma])=\rho(\gamma)\tau^{\rho}(X,Y,\xi)\in K_1(R),
\]
where $[\gamma]\in H_1(X;\Z)$ is represented by $\gamma$.
\end{lemma}

\subsection{Homology cylinders and the $Y$-filtration}
\label{sec:homologycylinders}
Let $\Sigma_{g,1}$ be a compact oriented surface of genus $g$ with connected boundary.
Let $M$ be a compact oriented 3-manifold and let $i\colon \partial(\Sigma_{g,1}\times [-1,1])\to \partial M$ be an orientation-preserving homeomorphism.
Two pairs $(M,i)$ and $(M',i')$ are said to be equivalent
if there is an orientation-preserving homeomorphism $f\colon M\to M'$ satisfying $f\circ i = i'$.
Let $i_{\pm}$ denotes the restriction of $i$ to $\Sigma_{g,1}\times\{\pm1\}$, respectively.
A \emph{homology cobordism} over $\Sigma_{g,1}$ is an equivalence class of a pair $(M,i)$ such that  $i_{\pm}$ induce isomorphisms $H_\Z\to H_1(M;\Z)$.
The set of homology cobordisms $\C$ over $\Sigma_{g,1}$ forms a monoid by stacking operation defined by $(M,i)\circ (N,j)=(M\cup_{i_+=j_-}N, i_-\cup j_+)$,
A homology cobordism $(M,i)$ is called a \emph{homology cylinder} if $i_+$ and $i_-$ induce the same isomorphism $H_\Z\to H_1(M;\Z)$.
We write $\I\C$ for the submonoid consisting of homology cylinders.

Compact 3-manifolds $M$ and $N$ are called \emph{$Y_d$-equivalent} if $N$ is obtained from $M$ by a series of surgeries along connected claspers of degree $d$,
and we denote it as $M\sim_{Y_d} N$.
As in \cite[Section~8.4.1]{Hab00C} (see \cite[Lemma~A.2]{Mas07}, for the proof),
it is equivalent to saying that $N$ is obtained from $M$ by a series of Torelli surgeries of class $d$,
that is, surgeries along embedded surfaces $\Sigma_{h,1}$ by elements of the $d$th term $\I_{h,1}(d)$ of the lower central series of the Torelli group $\I_{h,1}$ (see Section~\ref{subsec:Johnson}).
Let $Y_n\I\C$ denote the submonoid of $\I\C$ consisting of homology cylinders that are $Y_n$-equivalent to the trivial one $\Sigma_{g,1}\times[-1,1]$.
This gives a descending series $\I\C = Y_1\I\C \supset Y_2\I\C \supset \cdots$, which is called the \emph{$Y$-filtration}.
It is known that the monoid $Y_n\I\C/Y_{n+1}$ consisting of the $Y_{n+1}$-equivalence classes is a finitely generated abelian group.

%%%
\subsection{The Johnson filtration and the Johnson homomorphism}
\label{subsec:Johnson}
For a group $G$,
we write $G(d)$ for the $d$th term of the lower central series of $G$, that is, 
\[
G(1)=G\text{ and }G(d)=[G(d-1),G].
\]
Let $(M,i) \in \C$ and $d\ge1$.
The embeddings $i_\pm\colon \Sigma_{g,1}\to M$ induce isomorphisms 
\[
(i_\pm)_\ast \colon \frac{\pi}{\pi(d+1)} \to \frac{\pi_1(M)}{\pi_1(M)(d+1)}
\]
by \cite[Theorem~5.1]{Sta65}.
This gives a monoid homomorphism 
\[
\C \to \Aut\left(\frac{\pi}{\pi(d+1)}\right)
\]
defined by $(M,i) \mapsto (i_-)_\ast^{-1}\circ(i_+)_\ast$,
and its kernel is denoted by $J_d\C$.
The series of submonoids 
\[
\C \supset J_1\C=\I\C \supset J_2\C \supset \cdots
\]
is called the \emph{Johnson filtration} on $\C$. 
See \cite[Section~3]{HaMa12}, for more details.
As in \cite[(5.2)]{HaMa12},
there is an inclusion $Y_d\I\C\subset J_d\C$.

Since $\pi$ is a free group,
there is a natural isomorphism
\[
\frac{\pi(d)}{\pi(d+1)}\cong L_d(H_\Z),
\]
where $L_d(H_\Z)$ is the degree $d$ part of the free Lie ring generated by $H_\Z$.
The $d$th Johnson homomorphism
\[
\tau_d\colon J_d\C\to \Hom(H,L_d(H_\Z))
\]
is defined by
\[
\tau_d(M,i)([x])=((i_-)_\ast^{-1}\circ(i_+)_\ast)(x)x^{-1}\in \frac{\pi(d+1)}{\pi(d+2)}
\]
for $x\in \pi$.
As in \cite{Joh83},
the Johnson filtration was originally defined for the mapping class group $\M$ of the surface $\Sigma_{g,1}$.
We also denote the $d$th subgroup of the filtration as $J_d\M=\Ker(\M \to \Aut(\pi/\pi(d+1)))$.

%%%%%%%%
\section{Two Reidemeister-Turaev torsions $\tilde{\alpha}$ and $\alpha$}
\label{section:reidemeistertorsion}
We set $\pi=\pi_1\Sigma_{g,1}$ and denote by $\widehat{\Q\pi}=\plim_n\Q\pi/I^n$ the $I$-adic completion of the rational group ring $\Q\pi$,
where $I$ is the augmentation ideal.
In this section, we construct a crossed homomorphism
\[
\tilde{\alpha}\colon\C\to K_1(\widehat{\Q\pi})/\tilde{\rho}(H_\Z),
\]
where $\tilde{\rho}\colon H_\Z\to K_1(\widehat{\Q\pi})$ is induced by the natural inclusion $\pi \hookrightarrow \widehat{\Q\pi}^\times$.
The restriction of $\tilde{\alpha}$ to $\I\C$ lifts to a homomorphism
\[
\tilde{\alpha}\colon\I\C\to K_1(\widehat{\Q\pi})
\]
which can be considered as a lift of the $K_1(Q(\Z H_\Z))$-valued Reidemeister-Turaev torsion in \cite{MaMe13}.

%%%
\subsection{The Reidemeister-Turaev torsion valued in $K_1(\widehat{\Q\pi})$}
For $(M, i)\in \C$,
let us denote $\partial_{\pm} M=i_{\pm}(\Sigma_{g,1})$.
First, we check that the chain complex $C_*(M,\partial_-M;\widehat{\Q\pi})$ is acyclic to obtain the Reidemeister torsion.
The following lemma is proved in the same way as \cite[Proposition~2.1]{KLW01}.
\begin{lemma}\label{lem:acylic1}
Let $X$ be a finite connected CW-complex and $Y$ a connected subcomplex satisfying $H_*(X,Y;\mathbb{Z})=0$.
Let $R$ be a local ring with maximal ideal $\mathfrak{m}$.
If a homomorphism $\rho\colon \pi_1X\to R^{\times}$ satisfies $\rho(\pi_1 X)=\{1\}\subset \mathfrak{m}\backslash R$,
the chain complex $C_*(X,Y;R)$ is acyclic.
\end{lemma}
\begin{proof}
Let us fix a basis of $C_*(X,Y;\Z)$ and denote by $p\colon \tilde{X}\to X$ the universal covering.
Lifting the basis to elements in $C_*(\tilde{X},p^{-1}(Y);\Z)$,
we obtain a basis of $C_*(X,Y;R)$ as a free right $R$-module.
Denoting a chain homotopy $D\colon C_*(X,Y;\Z)\to C_{*+1}(X,Y;\Z)$,
we have 
\[
D\partial+\partial D=\id \colon C_*(X,Y;\Z)\to C_*(X,Y;\Z).
\]
We define a right $R$-module homomorphism
\[
\tilde{D}\colon C_*(X,Y;R)\to C_*(X,Y;R)
\]
by $\tilde{D}(\tilde{e})=\sum q_i\tilde{f}_i$ if $D(e)=\sum q_if_i$,
where $\tilde{e},\tilde{f}_i\in C_*(\tilde{X},p^{-1}(Y);\Z)$ are the chosen lifts of $e,f_i$ in the basis of $C_*(X,Y;\Z)$, respectively.
Let us set
\[
\Phi=\tilde{D}\tilde{\partial}+\tilde{\partial} \tilde{D}\colon C_*(X,Y;R)\to C_*(X,Y;R),
\]
where $\tilde{\partial}\colon C_*(X,Y;R)\to C_{*-1}(X,Y;R)$ is the boundary operator.
Since $\pi_1X$ acts on $\mathfrak{m}\backslash R$ from the right trivially,
the induced endomorphism on $C_*(X,Y;\mathfrak{m}\backslash R)$ by $\Phi$ is the identity map.
As in the proof of Proposition~2.2.4 in \cite{Ros94},
the matrix representing $\Phi\colon C_*(X,Y;R)\to C_*(X,Y;R)$ is invertible
since its reduction to $\mathfrak{m}\backslash R$ is invertible.
Thus, we conclude that $C_*(X,Y;R)$ is acyclic.
\end{proof}

Consider the case when $X=M$ and $Y=\partial_-M$ for $M\in\C$.
Choosing a triangulation of the pair $(M, \partial_-M)$,
we treat it as a pair of finite CW-complexes.
If we have an acyclic chain complex $C_*(M, \partial_-M;R)$ with local system $\rho\colon \pi_1M\to R^{\times}$,
we obtain the torsion $\tau^\rho(M,\partial_- M)\in K_1(R)/\rho(H_\Z)$,
where we denote by the same symbol $\rho$ the induced homomorphism
\[
H_\Z \xrightarrow[\cong]{(i_{-})_\ast} H_1(M;\Z)\to K_1(R).
\]
If we fix an equivalence class $\xi=\tilde{E}$ of families of cells,
we obtain the Reidemeister-Turaev torsion $\tau^\rho(M,\partial_-M,\xi)\in K_1(R)$ as in Section~\ref{section:RT-torsion}.

In \cite[Definition~3.8]{MaMe13},
Massuyeau and Meilhan constructed a canonical Euler structure called the preferred Euler structure for $M\in\I\C$,
and it gives a canonical equivalence class which we denote by $\xi_0\in\widehat{\Eul}(M,\partial_-M)$.
Let 
\[
\tilde{\rho}\colon \Z\pi_1M\to \widehat{\Q\pi_1M}\cong \widehat{\Q\pi},
\]
where the last isomorphism is induced by the embedding $i_-\colon \Sigma_{g,1}\to M$  using \cite[Theorem~5.1]{Sta65} and \cite{Qui68}.
See also \cite[Theorem~A.5]{HaMa09}. 
Let 
\[
\rho\colon \Z\pi_1M\to \Z(H_1(M;\Z))\cong \Z H_\Z\to Q(\Z H_\Z),
\]
where $Q(\Z H_\Z)$ is the field of fractions of $\Z H_\Z$,
and the second isomorphism is induced by $i_-$.
For $M\in\I\C$,
let us denote
\begin{align*}
\tilde{\alpha}(M)&=\tau^{\tilde{\rho}}(M,\partial_-M;\xi_0)\in K_1(\widehat{\Q\pi}),\\
\alpha(M)&=\tau^{\rho}(M,\partial_-M;\xi_0)\in K_1(Q(\Z H_\Z)).
\end{align*}
For $M\in\C$,
we also denote 
\[
\tau^{\tilde{\rho}}(M,\partial_-M)\in K_1(\widehat{\Q\pi})/\tilde{\rho}(H_\Z)\text{ and }
\tau^{\rho}(M,\partial_-M)\in K_1(Q(\Z H_\Z))/\rho(H_\Z)
\]
by the same symbols $\tilde{\alpha}$ and $\alpha$, respectively.

\begin{remark}
The mapping class group $\M$ of the surface $\Sigma_{g,1}$ is embedded into $\C$.
As in \cite[Example~3.9]{MaMe13},
$\alpha(\varphi)=0\in K_1(Q(\Z H_\Z))/\rho(H_\Z)$ for $\varphi\in\M$.
In the same way, we see that $\tilde{\alpha}(\varphi)=0\in K_1(\widehat{\Q\pi})/\tilde{\rho}(H_\Z)$.
\end{remark}

By \cite[Proposition~4.1]{FJR11} or \cite[Proposition~3(1)]{GoSa13},
the torsion $\alpha(M)\in K_1(Q(\Z H_\Z))$ under the identification $\det\colon K_1(Q(\Z H_\Z))\cong Q(\Z H_\Z)^\times$ is equal to the determinant of a square matrix with coefficients in $\Z H_\Z$.
Thus, we may consider that $\alpha(M)$ for $M\in\I\C$ lies in $\Z H_\Z$.
See also \cite[Proposition~3.5]{MaMe13}.
By \cite[Theorem~1.1.1 and Lemma~1.11.4]{Tur86},
the image of $\alpha\colon \I\C\to\Z H_\Z$ is contained in $(1+I H_\Z)\cup (-1+IH_\Z)$,
and it is actually contained in $1+I H_\Z$ since we eliminated the sign indeterminacy. 
Let $\widehat{\Q H_\Z}$ denote the completion of the group ring $\Q H_\Z$ with respect to the augmentation ideal $IH_\Z$.

\begin{proposition}\label{prop:compare-torsions}
The diagram
\[
\xymatrix{
\I\C \ar[r]^-{\tilde{\alpha}} \ar[d]_-{\alpha} & K_1(\widehat{\Q\pi}) \ar[d]\\
1+I H_\Z\ar[r] & K_1(\widehat{\Q H_\Z})
}
\]
commutes,
where the bottom homomorphism $1+IH_\Z\to \widehat{\Q H_\Z}^\times\cong K_1(\widehat{\Q H_\Z})$ is the natural inclusion, and the right vertical homomorphism is the induced map by the natural homomorphism $\widehat{\Q\pi}\to\widehat{\Q H_\Z}$.
\end{proposition}
\begin{proof}
Let us denote by $(\Q H_\Z)\Gamma^{-1}$ the localization of $\Q H_\Z$ with respect to the multiplicatively closed set $\Gamma=1+IH_\Z\subset \Q H_\Z$,
and let
\begin{align*}
&\rho'\colon\Z\pi_1M\to \Z H_1(M;\Z)\cong \Z H_\Z \to \widehat{\Q H_\Z},\\
&\rho''\colon \Z\pi_1M\to \Z H_1(M;\Z) \cong\Z H_\Z\to (\Q H_\Z)\Gamma^{-1}.
\end{align*}
Since $\widehat{\Q H_\Z}$ and $(\Q H_\Z)\Gamma^{-1}$ are local rings,
we obtain the Reidemeister-Turaev torsions
\begin{align*}
\alpha'(M) &= \tau^{\rho'}(M,\partial_-M,\xi_0)\in K_1(\widehat{\Q H_\Z}), \\
\alpha''(M) &= \tau^{\rho''}(M,\partial_-M,\xi_0)\in K_1((\Q H_\Z)\Gamma^{-1})
\end{align*}
by Lemma~\ref{lem:acylic1}.
We show that two diagrams
\[
\xymatrix{
\I\C \ar[r]^-{\tilde{\alpha}} \ar[rd]_-{\alpha'} & K_1(\widehat{\Q\pi}) \ar[d]\\
 & K_1(\widehat{\Q H_\Z}),
}
\qquad
\xymatrix{
\I\C \ar[d]_-{\alpha} \ar[rd]^-{\alpha'} &\\
1+IH_\Z \ar[r]& K_1(\widehat{\Q H_\Z})
}
\]
commute.
The first diagram commutes by the naturality of the Reidemeister-Turaev torsion.
Let us denote by $\xi\colon (\Q H_\Z)\Gamma^{-1}\to Q(\Z H_\Z)$ the composition of the natural homomorphism $(\Q H_\Z)\Gamma^{-1}\to Q(\Q H_\Z)$ and the isomorphism $Q(\Q H_\Z)\cong Q(\Z H_\Z)$.
It induces an injective homomorphism $\xi_*\colon K_1((\Q H_\Z)\Gamma^{-1})\to K_1(Q(\Z H_\Z))$ since $K_1((\Q H_\Z)\Gamma^{-1})\cong ((\Q H_\Z)\Gamma^{-1})^\times$ and $\Q H_\Z$ is an integral domain.
The homomorphism $\xi'\colon (\Q H_\Z)\Gamma^{-1}\to \widehat{\Q H_\Z}$ defined by $f/g\mapsto fg^{-1}$ also induces an injective homomorphism $\xi'_*\colon K_1((\Q H_\Z)\Gamma^{-1})\to K_1(\widehat{\Q H_\Z})$.
As a consequence, we have a commutative diagram
\[
\xymatrix{
&\I\C \ar[ld]_-{\alpha} \ar[rd] \ar[d]^-{\alpha''} \ar[rd]^-{\alpha'}& \\
K_1(Q(\Z H_\Z))& K_1((\Q H_\Z)\Gamma^{-1})\ar[l]^{\xi_*}  \ar[r]_{\quad \xi'_*} &K_1(\widehat{\Q H_\Z}).
}
\]
Since $K_1((\Q H_\Z)\Gamma^{-1})\to K_1(Q(\Z H_\Z))$ is injective,
the homomorphism $\alpha''\colon \I\C\to K_1((\Q H_\Z)\Gamma^{-1})$ lifts to $\alpha\colon \I\C\to 1+I H_\Z$
with respect to the injective homomorphism $1+I H_\Z \to K_1((\Q H_\Z)\Gamma^{-1})$
defined by $f\mapsto f/1$,
and we see that the second diagram commutes.
\end{proof}

\subsection{$\tilde{\alpha}$ is a crossed homomorphism}
\label{subsec:crossed_hom}
Given a homology cylinder $(M,i)\in\C$,
we denote the automorphism
\[
\widehat{\Q\pi}\xrightarrow{i_+} \widehat{\Q\pi_1(M)} \xrightarrow{i_-^{-1}}\widehat{\Q\pi}
\]
by $\sigma_M$.
It also induces an automorphism on $K_1(\widehat{\Q\pi})$.
Here, we prove:
\begin{proposition}\label{prop:crossed homo}
The map $\tilde{\alpha}\colon \C\to K_1(\widehat{\Q\pi})/\tilde{\rho}(H_\Z)$ is a crossed homomorphism, that is,
\[
\tilde{\alpha}(M\circ N)=\tilde{\alpha}(M)\cdot(\sigma_M)_*\tilde{\alpha}(N)\in K_1(\widehat{\Q\pi})/\tilde{\rho}(H_\Z).
\]
Moreover,  $\tilde{\alpha}\colon \I\C\to K_1(\widehat{\Q\pi})$ is also a crossed homomorphism. 
\end{proposition}
Note that $\alpha\colon \I\C\to K_1(Q(\Z H_\Z))$ is shown to be a crossed homomorphism in \cite[Lemma~3.14]{MaMe13}.
To prove Proposition~\ref{prop:crossed homo},
we need the following lemma.

\begin{lemma}\label{lem:euler-str}
For $M\in \I\C$,
\[
\tilde{\alpha}(M) = \tilde{\rho}\left(-\frac{1}{2}(C\circ\tau_1)(M)\right) \in K_1(\widehat{\Q\pi}/\hat{I}^2),
\]
where $\tilde{\rho}\colon H_\Z\to K_1(\widehat{\Q\pi}/\hat{I}^2)$ is the induced homomorphism by $\tilde{\rho}\colon \pi_1M\to \widehat{\Q\pi}/\hat{I}^2$.
\end{lemma}

\begin{proof}
In our convention for chain complexes,
we have
\[
\alpha(M)=\rho\left(-\frac{1}{2}(C\circ \tau_1)(M)\right)\in (1+IH_\Z)/(IH_\Z)^2
\]
by \cite[Lemma~3.12]{MaMe13}.
By the commutative diagrams in the proof of Proposition~\ref{prop:compare-torsions},
we have
\[
\alpha'(M)=\rho'\left(-\frac{1}{2}(C\circ \tau_1)(M)\right)\in K_1(\widehat{\Q H_\Z}/(\widehat{IH_\Z})^2).
\]
The determinant gives an isomorphism
\[
K_1(\widehat{\Q H_\Z}/(\widehat{IH_\Z})^2)\cong (\widehat{\Q H_\Z}/(\widehat{IH_\Z})^2)^\times,
\]
and it is isomorphic to the abelian group $(\widehat{\Q \pi}/\hat{I}^2)^\times$,
where the isomorphism is induced by the natural homomorphism $\Q\pi\to\Q H_\Z$.
Thus, we have an isomorphism
\begin{equation}
\label{eq:k1isom}
K_1(\widehat{\Q\pi}/\hat{I}^2)
\cong K_1(\widehat{\Q H_\Z}/(\widehat{IH_\Z})^2),
\end{equation}
and we obtain
\[
\tilde{\alpha}(M)=\tilde{\rho}\left(-\frac{1}{2}(C\circ\tau_1)(M)\right)\in K_1(\widehat{\Q \pi}/\hat{I}^2)
\]
by the commutative diagrams in the proof of Proposition~\ref{prop:compare-torsions}.
\end{proof}

\begin{proof}[Proof of Proposition~\ref{prop:crossed homo}]
Let $M,N\in \C$.
Similar to \cite[Proposition 3.5]{CFK11},
we obtain
\[
\tilde{\alpha}(M\circ N)
=\tilde{\alpha}(M)
\cdot(\sigma_M)_*\tilde{\alpha}(N)\in K_1(\widehat{\Q\pi})/{\pm}\tilde{\rho}(H_\Z).
\]
See also \cite[Section~7]{Mil66}, \cite[Proposition~6.6]{Sak08}, and \cite[Proposition~3.10]{Kit12} for related results.
Thus, there exists $\epsilon\in\{\pm1\}$ such that
\[
\tilde{\alpha}(M\circ N)
=\epsilon\cdot \tilde{\alpha}(M)
\cdot(\sigma_M)_*\tilde{\alpha}(N)\in K_1(\widehat{\Q\pi})/\tilde{\rho}(H_\Z).
\]
If we map it by the induced map $K_1(\widehat{\Q\pi})\to K_1(\Q)$ by the augmentation map $\widehat{\Q\pi}\to\Q$,
we see that $\epsilon=1$ by Proposition~\ref{prop:compare-torsions}.
Next, suppose that $M,N\in\I\C$.
We see that 
\[
\tilde{\alpha}(M\circ N)
=\tilde{\rho}(h)\cdot \tilde{\alpha}(M)
\cdot(\sigma_M)_*\tilde{\alpha}(N)\in K_1(\widehat{\Q\pi})
\]
for some $h\in H_\Z$.
By the equality~(\ref{eq:k1isom}),
the homomorphism $\sigma_M$ induces the identity map on $K_1(\widehat{\Q \pi}/\hat{I}^2)$.
Thus, Lemma~\ref{lem:euler-str} implies that
\[
1 = \tilde{\rho}(h)\in K_1(\widehat{\Q\pi}/\hat{I}^2)\cong (\widehat{\Q H_\Z}/(\widehat{IH_\Z})^2)^\times,
\]
and we have $1=\tilde{\rho}(h) \in K_1(\widehat{\Q\pi})$ by Corollary~\ref{cor:gr} below.
\end{proof}

\subsection{Computing Reidemeister-Turaev torsions}
\label{section:compute-RT}
In \cite{GoSa13},
Goda and \mbox{Sakasai} described the Reidemeister torsion from a presentation of the fundamental group of a homology cylinder.
See also \cite[Proposition~4.1]{FJR11}, for a similar result in sutured 3-manifolds.
Here, we apply their results to our torsion $\tilde{\alpha}\colon \C\to K_1(\widehat{\Q\pi})/\tilde{\rho}(H_\Z)$.

Let $(M,i)\in\C$.
Recall that we denote by $i_{\pm}\colon\Sigma_{g,1}\to M$ the restriction of $i$ to the top and bottom surfaces $\Sigma_{g,1}\times\{\pm 1\}$, respectively.
A presentation $P$ of $\pi_1M$ of the form
\[
\braket{
i_+(\gamma_1),\ldots, i_+(\gamma_{2g}),z_1,\ldots,z_l,i_-(\gamma_1),\ldots,i_-(\gamma_{2g})\mid
r_1,\ldots,r_{2g+l}}
\]
for some integer $l\ge0$ is called balanced.
For a balanced presentation $P$, 
let us consider the matrices
\[
A=\left(
\overline{\frac{\partial r_j}{\partial i_+(\gamma_i)}}
\right)_{\substack{1\le i\le 2g\\1\le j\le 2g+l}}
\text{ and }
B=\left(
\overline{\frac{\partial r_j}{\partial z_i}}
\right)_{\substack{1\le i\le l\\1\le j\le 2g+l}},
\]
where each entry is defined by Fox's free derivative and $\bar{v}\in \Z\pi_1M$ is obtained from $v\in \Z\pi_1M$ by applying the involution induced from $x\mapsto x^{-1}$ for $x\in\pi_1M$.
We denote
\[
A(P)=
\begin{pmatrix}
A\\
B
\end{pmatrix}.
\]

\begin{remark}
The matrix $A(P)$ corresponds to the transpose of the Alexander matrix in \cite[Definition~16.4]{Tur01} after taking the bar of each entry.
This is due to the convention for chain complexes we adopted.
\end{remark}

\begin{proposition}[{\cite[Proposition~3(1)]{GoSa13}}]\label{prop:torsion-of-presentation}
Let $\Gamma$ be a poly-torsion-free abelian group, 
and let  $\rho\colon\Z\pi_1M\to Q(\Z\Gamma)$ be a homomorphism induced by a group homomorphism $\pi_1 M\to \Gamma$.
If $P$ is a balanced presentation of $\pi_1M$,
we have
\[
\tau^\rho(M,\partial_- M)=\rho(A(P))
\in K_1(Q(\Z\Gamma))/{\pm}\rho(H_\Z).
\]
\end{proposition}

Here we recall its proof to obtain a corresponding result in our situation.
%In the proof,
Goda and Sakasai showed that a square matrix with coefficients in $Q(\Z\Gamma)$ is invertible when its reduction to $\Z$-coefficients is invertible.
The argument holds if we replace $Q(\Z\Gamma)$ and $\Z$ by a local ring $R$ and a skew field $\mathfrak{m}\backslash R$, respectively,
as in the proof of Proposition~2.2.4 in \cite{Ros94}.
Moreover, in the proof of the equality in \cite[Proposition~3(1)]{GoSa13}, the case where a $3$-manifold $M$ is irreducible is easy to treat.
In the case where $M$ is reducible, they need the fact that any homomorphism from a group whose abelianization is finite to a poly-torsion-free abelian group is trivial.
We replace this argument by the fact that for a complete augmented algebra $R$ over $\Q$ (see \cite[Appendix~A.1]{Qui69} for the definition), any homomorphism from a perfect group to $R^\times$ is trivial.
Note that the augmented ideal $I$ is the maximal ideal $\mathfrak{m}$ for such an $R$.
Hence, we obtain the following.

\begin{lemma}\label{lem:goda-sakasai}
%Let $R$ be a local ring with the maximal ideal $\mathfrak{m}$,
Let $R$ be a complete augmented algebra over $\Q$
and let $\rho\colon\Z\pi_1M\to R$ be a homomorphism satisfying $\rho(\pi_1M)\subset 1+I$.
If $P$ is a balanced presentation of $\pi_1M$,
we have
\[
\tau^{\rho}(M,\partial_-M)=\rho(A(P))
\in K_1(R)/{\pm} \rho(H_\Z).
\]
\end{lemma}

Furthermore, there exists a balanced presentation such that the equality in Lemma~\ref{lem:goda-sakasai} holds as elements of $K_1(R)$.

\begin{lemma}\label{lem:torsion-of-presentation}
Let $R$ and $\rho\colon\Z\pi_1M\to R$ be as above.
For any Euler structure $\xi$ of $M\in\C$,
there exists a balanced presentation $P$ such that
\[
\tau^\rho(M,\partial_-M,\xi)=\rho(A(P))
\in K_1(R).
\]
\end{lemma}

\begin{proof}
Let $P$ be a presentation of $\pi_1(M)$ as above.
If we change the order of the relators $r_1,\ldots, r_{2g+l}$,
we may assume
$\epsilon(\rho(A(P)))=1\in K_1(\Q)$.
Combining with \cite[Lemma~3.12]{MaMe13} and Proposition~\ref{prop:compare-torsions},
we obtain
\[
\tau^\rho(M,\partial_-M)=\rho(A(P))
\in K_1(R)/\rho(H_\Z).
\]
Moreover, if we change the relator $r_i$ to $sr_is^{-1}$ for some generator $s$ and $1\le i\le 2g+l$,
$\rho(A(P))$ is multiplied by $\rho([s])$ represented by the generator $s$.
Thus, we can change $P$ into the one satisfying
\[
\tau^\rho(M,\partial_-M,\xi)=\rho(A(P))
\in K_1(R).
\]
\end{proof}

%%%%
\section{The $K_1$-group of $\widehat{\Q\pi}$}
\label{sec:K1_of_Qpi}
This section is devoted to showing Theorem~\ref{thm:K1Qpihat} concerning $\widehat{\Q\pi}$.
We first discuss a more general situation and prove some propositions and lemmas below.
For a unital associative algebra $R$ over $\Q$ with a surjective homomorphism $\epsilon\colon R \twoheadrightarrow \Q$, let $I=\Ker\epsilon$ and $\hat{R}=\varprojlim_n R/I^n$.
Define $\log \colon 1+\hat{I} \to \hat{I}$ and $\exp \colon \hat{I} \to 1+\hat{I}$ by
\[
\log(1+x) = \sum_{n=1}^\infty \frac{(-1)^{n-1}}{n}x^n,\quad \exp(x) = \sum_{n=1}^\infty \frac{1}{n!}x^n,
\]
respectively.
Since $\log$ and $\exp$ preserve the filtration, they are continuous with respect to the $I$-adic topology.

\begin{proposition}
\label{prop:log-exp}
The maps $\log$ and $\exp$ induce isomorphisms which fit into a commutative diagram
\begin{equation}
\label{eq:log-exp}
\vcenter{\xymatrix{
H_1(1+\hat{I}) \ar@{->>}[r] \ar@<-0.2em>[d]_-\log & \dfrac{1+\hat{I}}{\cl[1+\hat{I},1+\hat{I}]} \ar[r]^-{\cong} \ar@<-0.2em>[d]_-\log & \varprojlim_n H_1(1+\hat{I}/\hat{I}^n) \ar@<-0.2em>[d]_-\log \\
H_1^\Lie(\hat{I}) \ar@{->>}[r] \ar@<-0.2em>[u]_-\exp & \dfrac{\hat{I}}{\cl[\hat{I},\hat{I}]} \ar[r]^-{\cong} \ar@<-0.2em>[u]_-\exp & \varprojlim_n H_1^\Lie(\hat{I}/\hat{I}^n) \ar@<-0.2em>[u]_-\exp ,}}
\end{equation}
\end{proposition}

\begin{proof}
First, the Baker-Campbell-Hausdorff formula shows that for any $x,y \in \hat{I}$ there exist $z,w \in \hat{I}$ such that
\[
\log(\exp(x)\exp(y)) = x+y+\frac{1}{2}[x,y]+\cdots = x+y+[x,z]+[y,w] \in \hat{I}.
\]
Thus, the map $\log$ induces a homomorphism $\log\colon H_1(1+\hat{I}) \to H_1^\Lie(\hat{I})$.
In particular, we have $\log([1+\hat{I},1+\hat{I}]) \subset [\hat{I},\hat{I}]$, and hence $\log(\cl[1+\hat{I},1+\hat{I}]) \subset \cl[\hat{I},\hat{I}]$.
It follows that $\log$ induces a homomorphism $\frac{1+\hat{I}}{\cl[1+\hat{I},1+\hat{I}]} \to \frac{\hat{I}}{\cl[\hat{I},\hat{I}]}$.
Also, $\log$ gives a homomorphism $H_1(1+\hat{I}/\hat{I}^n) \to H_1^\Lie(\hat{I}/\hat{I}^n)$.

Next, by the Baker-Campbell-Hausdorff formula, we see that for any $x,y \in \hat{I}$ there exist $z,w \in \hat{I}$ such that
\[
\log(\exp(-y)\exp(-x)\exp(x+y)) = [x,z]+[y,w] \in \hat{I}.
\]
Lemma~\ref{lem:SympExp} below implies that $\exp(-y)\exp(-x)\exp(x+y)$ lies in the commutator subgroup $[1+\hat{I},1+\hat{I}]$, and hence the map $\exp$ induces a homomorphism $\exp\colon \hat{I} \to H_1(1+\hat{I})$.
This map factors through $H_1^\Lie(\hat{I})$ by Lemma~\ref{lem:SympExp} again.
The same argument as above shows that $\exp$ derives homomorphisms $\frac{\hat{I}}{\cl[\hat{I},\hat{I}]} \to \frac{1+\hat{I}}{\cl[1+\hat{I},1+\hat{I}]}$ and $H_1^\Lie(\hat{I}/\hat{I}^n) \to H_1(1+\hat{I}/\hat{I}^n)$.
Now, these three homomorphisms induced by $\exp$ are respectively the inverse maps of the homomorphisms induced by $\log$, and hence all the vertical maps in the diagram \eqref{eq:log-exp} are isomorphisms.

We next prove that the map $\frac{\hat{I}}{\cl[\hat{I},\hat{I}]} \to \varprojlim_n H_1^\Lie(\hat{I}/\hat{I}^n)$ in \eqref{eq:log-exp} is an isomorphism.
Then the commutativity implies that $\frac{1+\hat{I}}{\cl[1+\hat{I},1+\hat{I}]} \to \varprojlim_n H_1(1+\hat{I}/\hat{I}^n)$ is an isomorphism as well.
Since the inverse system $\{[\hat{I}/\hat{I}^{n+1}, \hat{I}/\hat{I}^{n+1}] \to [\hat{I}/\hat{I}^n, \hat{I}/\hat{I}^n]\}_n$ satisfies the Mittag-Leffler condition, we have an exact sequence
\[
0 \to \varprojlim_n [\hat{I}/\hat{I}^n, \hat{I}/\hat{I}^n] \to \hat{I} \to \varprojlim_n H_1^\Lie(\hat{I}/\hat{I}^n) \to 0.
\]
Here a natural homomorphism $\cl[\hat{I},\hat{I}] \to [\hat{I}/\hat{I}^n, \hat{I}/\hat{I}^n]$ induces $f\colon \cl[\hat{I},\hat{I}] \to \varprojlim_n [\hat{I}/\hat{I}^n, \hat{I}/\hat{I}^n]$.
If $f$ is an isomorphism, then we conclude that $\frac{\hat{I}}{\cl[\hat{I},\hat{I}]} \to \varprojlim_n H_1^\Lie(\hat{I}/\hat{I}^n)$ is an isomorphism.
The injectivity of $f$ follows from $\Ker f \subset \bigcap_{n=1}^\infty \hat{I}^n = \{0\}$.
Let $\{x_n\}_n \in \varprojlim_n [\hat{I}/\hat{I}^n, \hat{I}/\hat{I}^n]$.
Take a lift $\tilde{x}_n \in [\hat{I}, \hat{I}]$ of $x_n$.
Then, by the definition of the inverse limit, $\tilde{x}_{n+1} = \tilde{x}_n$ in $\hat{I}/\hat{I}^n$.
Therefore, $\{\tilde{x}_n\}_n$ defines an element of $\hat{I}$.
Moreover, it is contained in $\cl[\hat{I},\hat{I}]$ since $\tilde{x}_n \in [\hat{I}, \hat{I}]$.
It follows that $f$ is surjective, and thus it is an isomorphism.
\end{proof}

\begin{lemma}
\label{lem:SympExp}
Let $k \geq 1$.
For $x_j, y_j \in \hat{I}$ $(j=1,\dots,k)$, $\exp\left(\sum_{j=1}^k[x_j,y_j]\right)$ is written by a product of $k$ commutators in $1+\hat{I}$.
\end{lemma}

\begin{proof}
In this proof, we write $\pi^{(k)}$ for the free group generated by $a_1,\dots,a_k$, $b_1,\dots,b_k$.
Let $\hat{T}^{(k)}$ denote the completed tensor algebra generated by $\pi^{(k)}_\ab\otimes\Q$ and define an ideal $\hat{T}_1^{(k)}$ by $\hat{T}_1^{(k)}=\prod_{n=1}^\infty (\pi^{(k)}_\ab\otimes\Q)^{\otimes n}$.
Choose a symplectic expansion $\theta\colon \pi^{(k)} \to 1+\hat{T}_1^{(k)}$, which was introduced in \cite[Definition~2.15]{Mas12} and another construction was given in \cite[Theorem~1.1]{Kun12}.
By definition, we have
\[
\theta\left(\prod_{j=1}^k a_j b_j a_j^{-1} b_j^{-1}\right) = \exp\left(\sum_{j=1}^k[a_j,b_j]\right) \in 1+\hat{T}_1^{(k)}.
\]
Define an algebra homomorphism $\phi\colon T_1^{(k)} \to \hat{I}^{(k)}$ by $\phi(a_j)=x_j$ and $\phi(b_j)=y_j$.
Since $\phi$ preserves the filtration, it induces an algebra homomorphism $\phi\colon \hat{T}^{(k)} \to \hat{R}$.
We obtain $\prod_{j=1}^k(\phi\circ\theta)(a_jb_ja_j^{-1}b_j^{-1}) = \exp\left(\sum_{j=1}^k[x_j,y_j]\right)$ from the above equality by applying $\phi$.
\end{proof}

Here, we define
\[
\widetilde{\log} \colon H_1(\hat{R}^{\times})
\to
\Q^\times \oplus H_1^\Lie(\hat{I})
\]
by $x \mapsto (\epsilon(x),\log(x \epsilon(x)^{-1}))$.

\begin{lemma}
\label{lem:log_tilde}
The homomorphism $\widetilde{\log}$ is an isomorphism.
\end{lemma}

\begin{proof}
Define an isomorphism
\[
f \colon \hat{R}^{\times} \to \Q^\times \times (1+\hat{I})
\]
by $f(x)=(\epsilon(x), x \epsilon(x)^{-1})$.
It induces an isomorphism $f_\ast \colon H_1(\hat{R}^{\times}) \to \Q^\times \oplus H_1(1+\hat{I})$.
Then we have $\widetilde{\log} = (\id_{\Q^\times}\oplus\log) \circ f_\ast$.
Since $\log$ is an isomorphism by Proposition~\ref{prop:log-exp}, so is $\widetilde{\log}$.
\end{proof}

Here, note that $\hat{R}$ is a local ring since the set of non-unit elements is the ideal $\hat{I}$.
Thus, the Dieudonn\'e determinant 
\[
\det\colon K_1(\hat{R})\to (\hat{R}^{\times})_{\ab}
\]
is defined.
It is known to be an isomorphism as in \cite[Corollary~2.2.6]{Ros94},
and the inverse map is induced by the natural map
\[
\hat{R}^\times\to  \GL(1,\hat{R})\to K_1(\hat{R}).
\]
For a reference to Dieudonn\'e determinant, see \cite[Chapter~2.2]{Ros94} and \cite[Theorem~1.3]{LiYu21} for example.
Also note that, in \cite[Theorem~2.2.5(f)]{Ros94}, it is erroneously claimed that $\det A=\det A^T$.
When $\hat{R}$ is a non-commutative local ring, it is not true in general.

On the other hand, in \cite[Example~A.4]{MaSa20} (see also \cite{Kri02}), the ``log-determinant'' map
\[
\ldet\colon K_1(\hat{R}) \to H_1^\Lie(\hat{I})
\]
is constructed.
We write $M(N,R)$ for the ring of $N\times N$ matrices over a ring $R$,
and denote by $\epsilon\colon \GL(N,\hat{R}) \to\GL(N,\Q)$ the map which applies $\epsilon\colon \hat{R} \twoheadrightarrow \Q$ to all entries.
Define the logarithm map as
\begin{equation}
\label{eq:log}
\log B=\sum_{n=1}^\infty\frac{(-1)^{n-1}}{n}(B-I_N)^n\in M(N,\hat{I})
\end{equation}
for $B \in \Ker(\epsilon\colon \GL(N,\hat{R})\to\GL(N,\Q))$.
The map $\ldet$ is a homomorphism defined by
\[
\ldet(A) =(\tr\circ\log)(A\epsilon(A)^{-1})
\]
for $A \in \GL(N,\hat{R})$.

\begin{proposition}
\label{prop:ldet}
The diagram
\[
\xymatrix{
K_1(\hat{R}) \ar[r]^-{\widetilde{\ldet}} \ar[d]_-{\det} & \Q^\times\oplus H_1^\Lie(\hat{I}) \\
\hat{R}^\times_{\ab} \ar[ru]_-{\widetilde{\log}} &
}
\]
commutes and all the maps are isomorphisms, where $\widetilde{\ldet}(A)=(\det\epsilon(A), \ldet(A))$.
\end{proposition}

\begin{proof}
First, the inverse of $\det$ is the map induced by the inclusion $\hat{R}^\times \hookrightarrow \GL(\hat{R})$.
Hence we have
\[
\widetilde{\ldet}({\det}^{-1}(x)) = \left(\epsilon(x), \sum_{n=1}^\infty\frac{(-1)^{n-1}}{n}(x\epsilon(x)^{-1}-1)^n\right) = \widetilde{\log}(x)
\]
for $x \in \hat{R}^\times_{\ab}$.
This implies that the diagram is commutative.
Finally, since $\widetilde{\log}$ is an isomorphism by Lemma~\ref{lem:log_tilde}, so is $\widetilde{\ldet}$.
\end{proof}

In the rest of this section, we focus on the case $R=\Q\pi$ and $\epsilon$ is the augmentation map.
Let $\hat{T}$ denote the completed tensor algebra generated by $H$.
If we choose a Magnus expansion $\theta\colon \pi\to 1+\hat{T}_1$, it induces an isomorphism $\widehat{\Q\pi} \to \hat{T}$ as shown in \cite[Theorem~1.3]{Kaw06}.
Here, for $n\geq 1$, let $\hat{T}_n$ denote the ideal of $\hat{T}$ consisting of elements whose lowest degrees are greater than or equal to $n$.

\begin{proof}[Proof of Theorem~\ref{thm:K1Qpihat}]
By Proposition~\ref{prop:ldet}, the map $\widetilde{\ldet}\colon K_1(\widehat{\Q\pi}) \to \Q^\times\oplus H_1^\Lie(\hat{I})$ is an isomorphism.
Also, the map $\theta$ induces an isomorphism $H_1^\Lie(\hat{I}) \to H_1^\Lie(\hat{T}_1)$.
Thus, it suffices to prove that a natural homomorphism
\[
\phi\colon H_1^\Lie(\hat{T}_1) \to \varprojlim_n H_1^\Lie(\hat{T}_1/\hat{T}_n) = \prod_{n=1}^{\infty} (H^{\otimes n})_{\Z_n}
\]
is an isomorphism.
Note that a similar description of $\hat{T}/{\cl[\hat{T},\hat{T}]}$ can be found in \cite[Lemma~4.3.4]{KaKu12}.
We now have a natural homomorphism $\prod_{n=1}^{\infty} H^{\otimes n} \to H_1^\Lie(\hat{T}_1)$.
Let us prove that it factors through $\prod_{n=1}^{\infty} (H^{\otimes n})_{\Z_n}$, and thus it gives the inverse of $\phi$.
Here, since
\[
\Ker(H^{\otimes n} \twoheadrightarrow (H^{\otimes n})_{\Z_n}) \subset \sum_{i+j=n} [\hat{T}_i,\hat{T}_j],
\]
we have
\[
\Ker\left( \prod_{n=1}^{\infty} H^{\otimes n} \twoheadrightarrow \prod_{n=1}^{\infty} (H^{\otimes n})_{\Z_n} \right) \subset \cl[\hat{T}_1,\hat{T}_1]
\]
as subspaces of $\hat{T}_1$.
Therefore, it suffices to see that $\cl[\hat{T}_1,\hat{T}_1] = [\hat{T}_1,\hat{T}_1]$.
Fix a basis $\{a_1,\dots,a_g,b_1,\dots,b_g\}$ of $H$.
For $h \in H$ and $x,y \in \hat{T}_1$, we have $[hx,y]=[h,xy]+[x,yh]$.
Using this equality inductively, one can check that $[\hat{T}_1,\hat{T}_1] = \sum_{i=1}^g \left([a_i,\hat{T}_1]+[b_i,\hat{T}_1]\right)$.
Furthermore, for $d\geq 2$, we have
\[
[\hat{T}_1,\hat{T}_1]\cap \hat{T}_d = \sum_{i=1}^g \left([a_i,\hat{T}_{d-1}]+[b_i,\hat{T}_{d-1}]\right).
\]
Indeed, any $z=[\hat{T}_1,\hat{T}_1]\cap \hat{T}_d$ is expressed as $z=\sum_{i=1}^g \left([a_i,x_i]+[b_i,y_i]\right)$ for $x_i,y_i \in \hat{T}_1$, and then, for $x'_i, y'_i \in \hat{T}_{d-1}$ obtained from $x_i,y_i$ by eliminating terms of degree less than $d-1$, the sum $\sum_{i=1}^g \left([a_i,x'_i]+[b_i,y'_i]\right)$ coincides with $z$.

Here, let $z \in \cl[\hat{T}_1,\hat{T}_1]$.
Then there is a sequence $\{z_n\}$ in $[\hat{T}_1,\hat{T}_1]$ satisfying $\lim_{n\to\infty}z_n=z$ and $z_{n+1}-z_{n} \in \hat{T}_{n+1}$.
It follows from the above equality that $z_{n+1}-z_{n}=\sum_{i=1}^g \left([a_i,x'_{n,i}]+[b_i,y'_{n,i}]\right)$ for some $x'_{n,i}, y'_{n,i} \in \hat{T}_{n}$.
Since $\sum_{n=1}^\infty x'_{n,i}$ and $\sum_{n=1}^\infty y'_{n,i}$ lie in $\hat{T}_1$ for each $i$, we conclude that $z \in [\hat{T}_1,\hat{T}_1]$.
\end{proof}

A Magnus expansion $\theta$ induces an isomorphism
\[
\theta_*\colon \Ker\left(H_1^{\Lie}(\hat{I}/\hat{I}^{n+1}) \to H_1^{\Lie}(\hat{I}/\hat{I}^n)\right) \cong (H^{\otimes n})_{\Z_n},
\]
and it does not depend on the choice of $\theta$.
Setting $R=\Q\pi/I^{n+1}$ and $R=\Q\pi/I^n$ in Proposition~\ref{prop:ldet},
we have the following.

\begin{corollary}
\label{cor:gr}
The homomorphism
\[
\ldet_n\colon \Ker\left(K_1(\widehat{\Q\pi}/\hat{I}^{n+1})\to K_1(\widehat{\Q\pi}/\hat{I}^{n})\right) \to (H^{\otimes n})_{\Z_n}
\]
defined by 
$\ldet_n(A)=(\theta_*\circ \tr \circ \log)(A\epsilon(A)^{-1})$ is an isomorphism,
where 
$A\in\GL(N,\widehat{\Q\pi}/\hat{I}^{n+1})$.
In particular,
$\ldet_1(A)=\theta_*(\tr(A\epsilon(A)^{-1}-I_N))$
for $A\in\GL(N,\widehat{\Q\pi}/\hat{I}^2)$.
\end{corollary}

%%%%%%%%%%%
\section{Finite-type invariants}
\label{section:finitetype}
The $I$-adic reduction of the Reidemeister-Turaev torsion $\alpha\colon \I\C\to\mathbb{Z}H_\Z$ is proved to be a finite-type invariant in \cite[Theorem~3.11]{MaMe13}.
In this section, we show Theorem~\ref{thm:Y-filtration},
which states that the reduction $\tilde{\alpha}\colon \I\C\to K_1(\widehat{\Q\pi}/\hat{I}^{d})$ of $\tilde{\alpha}$ is also a finite-type invariant for $d\ge1$.

\subsection{Torelli surgeries and Euler structures}
Using Lemma~\ref{lem:torsion-of-presentation},
we prove that the $I$-adic reduction of our torsion is a finite-type invariant of homology cylinders.

Let $M\in \I\C$,
and let $V_1,\ldots, V_d$ denote disjoint handlebodies embedded in $M$.
Pick points $p$ in $\partial_-M$ and $q_k$ in $\partial V_k$ for $1\le k\le d$,
and choose paths $\alpha_1, \dots, \alpha_d$ in $M\setminus\bigcup_{k=1}^d\Int V_k$ such that $\alpha_k$ runs from $p$ to $q_k$ and $\Int\alpha_k$ are mutually disjoint.
We denote by $\{\gamma_i\}_{i=1}^{2g}$ the standard basis of $\pi_1\Sigma_{g,1}$ in Figure~\ref{fig:spine_gamma}.
We also denote by $x_{k,1},\dots,x_{k,2h_k}$ loops obtained by connecting the standard basis of $\pi_1(\partial V_k)$ and the path $\alpha_k$ for $1\le k\le d$,
where $h_k$ is the genus of $V_k$.
Set $M'=M\setminus \bigcup_{k=1}^d\Int(V_k\cup N(\alpha_k))$.
\begin{figure}[h]
 \centering
 \includegraphics[width=0.7\textwidth]{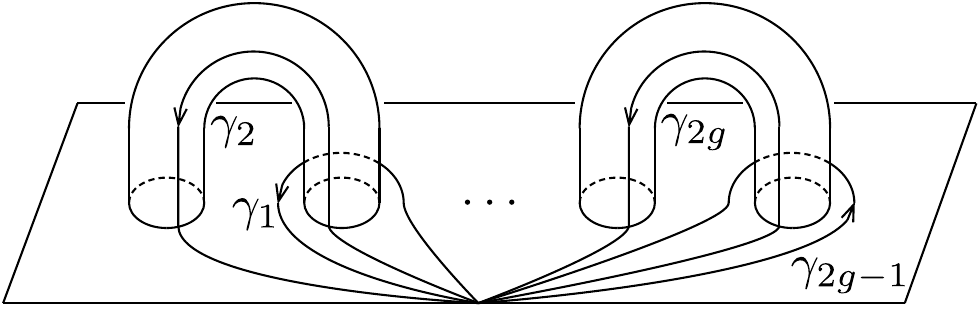}
 \caption{Oriented simple closed curves $\gamma_1,\dots,\gamma_{2g}$ on $\Sigma_{g,1}$.}
 \label{fig:spine_gamma}
\end{figure}

\begin{lemma}\label{lem:presentation-of-complement}
The fundamental group $\pi_1(M')$ has a presentation $P_0$ with generating set 
\[
\{i_-(\gamma_1),\ldots,i_-(\gamma_{2g})\}\cup\{x_{k,1},\ldots, x_{k,2h_k}\mid 1\le k\le d\}\cup \{y_1,\ldots, y_l\}
\]
and some relators $s_1,s_2,\ldots,s_{l+\sum_{k=1}^dh_k}$.
\end{lemma}

\begin{proof}[Sketch of the proof]
The proof proceeds in the same way as \cite[Lemma~4]{GoSa13},
and we only sketch the proof.
First, take a ``suitable'' triangulation of $\partial M'$, and extend it to $N$.
Shrinking all 3-simplices into the 2-skeleton one after another and taking the quotient by a maximal tree of the 1-skeleton,
we obtain a 2-dimensional CW-complex having only one vertex.
By the Mayer-Vietoris exact sequence,
we can compute the Euler number
\[
\chi(M') = 1-\left(2g+\sum_{k=1}^d h_k\right).
\]
Thus, if we denote the number of 1-simplices in the CW-complex by $2g+2\sum_{k=1}^dh_k+l$,
the number of 2-simplices is $\sum_{k=1}^dh_k+l$.
The desired presentation is obtained from this CW-complex.
\end{proof}

We may assume that the loops $x_{k,1},\ldots, x_{k,h_k}$ in $\partial V_k$ bound disjoint disks $D_{k,1},\ldots, D_{k,h_k}$ in $V_k$ which cut $V_k$ into a 3-ball.
We also assume that $x_{k,1},\ldots, x_{k,2h_k}$ represent a symplectic basis of $H_1(\partial V_k;\Z)$,
namely,
\[
[x_{k,i}]\cdot [x_{k,j}]=
\begin{cases}
1&\text{if }i-j=-h_k,\\
-1&\text{if }i-j=h_k,\\
0&\text{otherwise},
\end{cases}
\]
for $1\le i\le 2h_k$ and $1\le j\le 2h_k$,
where $x\cdot y$ denotes the intersection number of $x,y\in H_1(\partial V_k;\Z)$.
The element $i_+(\gamma_i) \in \pi_1(M')$ for $1\le i\le 2g$ is represented as a product of generators of $P_0$.
Add generators $i_+(\gamma_1), \ldots, i_+(\gamma_{2g})$ and such relators to the presentation $P_0$.
We obtain a balanced presentation $P$ of $\pi_1(M)$ in Proposition~\ref{prop:torsion-of-presentation} by adding relators $x_{k,1},\ldots, x_{k,h_k}$ for $1\le k\le d$ to $P_0$.
By Proposition~\ref{prop:torsion-of-presentation},
we have 
\[
\epsilon(\tilde{\rho}(A(P)))=\pm1\in K_1(\Q),
\]
and by changing the order of generators or relators,
we may assume $\epsilon(\tilde{\rho}(A(P)))=1$.

In Theorem~\ref{thm:Y-filtration},
for a subset $J\subset\{1,2,\ldots,d\}$,
we denote
\[
M_J=\left(M\setminus \bigsqcup_{i\in J} \Int V_i\right) \cup_{\{\varphi_i\}_{i\in J}} \bigsqcup_{i\in J}V_i,
\]
where $\varphi_i$ is a mapping class of $\partial V_i$ which acts on $H_1(\partial V_i;\Z)$ trivially.
Let $N(q_k)=\partial V_k\cap N(\alpha_k)$, which is a neighborhood of $q_k\in \partial V_k$.
Denote by $\pi_0\Diff_+(\partial V_k, N(q_k))$ the mapping class group of $\partial V_k$ which fixes $N(q_k)$ pointwise,
and let 
\[
\I(\partial V_k,N(q_k))=\Ker\bigl(\pi_0\Diff_+(\partial V_k, N(q_k))\to \Aut H_1(\partial V_k;\Z)\bigr)
\]
be its Torelli subgroup.
Choose a lift of $\varphi_k$ to an element in $\I(\partial V_k,N(q_k))$.
Then, we obtain a presentation $P_J$ of $\pi_1(M_J)$ of the form in Lemma~\ref{lem:presentation-of-complement} by adding relators $c_k(x_{k,1}),\ldots, c_k(x_{k,h_k})$ to $P_0$ for $1\le k\le d$,
where $c_k$ is $\id_{\partial V_k}$ if $k\notin J$, and $\varphi_k$ if $k\in J$.
Let 
\[
\tau_1\colon \I(\partial V_k,N(q_k))\to \Lambda^3 H_1(\partial V_k;\Z)
\]
denote the first Johnson homomorphism,
and define a contraction map $C_k\colon \Lambda^3H_1(\partial V_k;\Z) \to 2H_1(\partial V_k;\Z)$ by
\[
C_k(x\wedge y\wedge z) = 2((x\cdot y)z+(y\cdot z)x+(z\cdot x)y).
\]

\begin{lemma}\label{lem:torsion-presentation}
Let $P$ and $P_J$ denote the presentations of $\pi_1 M$ and $\pi_1M_J$ as above, respectively.
We have
\[
\tilde{\rho}(A(P_J))
=\tilde{\rho}\left(-\frac{1}{2}\sum_{k\in J}((\incl_k)_*\circ C_k\circ\tau_1)(\varphi_k)\right)\tilde{\rho}(A(P))
\in K_1(\widehat{\Q\pi}/\hat{I}^2),
\]
where $(\incl_k)_*\colon H_1(\partial V_k;\Z)\to H_1(M;\Z)$ is the induced homomorphism by the inclusion, and the first $\tilde{\rho}$ in the right-hand side is the induced map by
$\tilde{\rho}\colon \pi_1(M)\to \widehat{\Q\pi}\twoheadrightarrow \widehat{\Q\pi}/\hat{I}^2$.
\end{lemma}

\begin{proof}
Let $N=2g+2\sum_{k=1}^dh_k+l$.
The matrices $A(P)$ and $A(P_J)$ are square matrices of size $N$.
We may assume that the relators are ordered so that the last $\sum_{k=1}^{d}h_k$ relators are
\[
c_1(x_{1,1}),c_1(x_{1,2}),\ldots, c_1(x_{1,h_1}),\ldots, c_d(x_{d,1}),c_d(x_{d,2}),\ldots, c_d(x_{d,h_d})
\]
and that the generators are ordered so that the last $2\sum_{k=1}^d h_k$ generators are
\begin{align*}
&x_{1,h_1+1},\ldots, x_{1,2h_1},x_{2,h_2+1},\ldots, x_{2,2h_2},\ldots, x_{d,h_d+1},\ldots, x_{d,2h_d},\\
&x_{1,1},\ldots,x_{1,h_1},x_{2,1},\ldots,x_{2,h_2},\ldots,x_{d,1},\ldots,x_{d,h_d}.
\end{align*}
Since the presentation $P_J$ differs from $P$ only in relators coming from the meridian disks $D_{k,i}$ of $V_k$ for $k\in J$,
we see that $\epsilon(A(P_J))=\epsilon(A(P))\in\GL(N,\Q)$.
Denote it by $\epsilon_0$.

The homology cylinder $M$ is homotopy equivalent to 
\[
M'\cup \Bigl(\bigsqcup_{k=1}^d\Bigl(\bigsqcup_{i=1}^{h_k} D_{k,i}\Bigr)\Bigr).
\]
The Mayer-Vietoris exact sequence for it implies that $H_1(M';\Z)$ is freely generated by the homology classes $[i_-(\gamma_1)],\ldots, [i_-(\gamma_{2g})]$ and $[x_{k,1}],\ldots, [x_{k,h_k}]$ for $1\le k\le d$.
Thus, every element in $H_1(M';\Z)$ is written as a linear combination of these homology classes.
In particular, we can write
\[
[x_{l,h_l+j}]=\sum_{i=1}^{2g}t'_{i,(l,j)}[i_-(\gamma_i)]+\sum_{k=1}^d\sum_{i=1}^{h_k} t_{(k,i),(l,j)}[x_{k,i}]\in H_1(M';\Z)
\]
for $1\le l\le d$ and $1\le j\le h_l$.
Set $N_1=2g+l$ and $N_2=\sum_{k=1}^dh_k$.
Note that the first $N_1+N_2$ columns of $\epsilon_0$ correspond to relations of $H_1(M';\Z)$.
Thus, after applying identical column addition transformations to the first $N_1+N_2$ columns of $\rho(A(P_J))$ and $\rho(A(P))$,
we can write
\[
\epsilon_0=
\begin{pmatrix}
I_{N_1}&O_{N_1,N_2}&O_{N_1,N_2}\\
O_{N_2,N_1}&I_{N_2}&O_{N_2,N_2}\\
S&T&I_{N_2}
\end{pmatrix}
\in\GL(N,\Q)
\]
for some matrices $S$ of size $N_2\times N_1$ and $T$ of size $N_2\times N_2$, where $O_{m,n}$ denotes the zero matrix of size $m\times n$.
More precisely, setting 
\[
T_{kl}=(t_{(k,i),(l,j)})_{\substack{1\le i\le h_k\\1\le j\le h_l}},
\]
we can write $T=(T_{kl})_{\substack{1\le k\le d\\1\le l\le d}}$.
If we glue 2-handles to $M$ along the loops $i_+(\gamma_2),i_+(\gamma_4),\ldots,i_+(\gamma_{2g})$, $i_-(\gamma_1),i_-(\gamma_3),\ldots,i_-(\gamma_{2g-1})$,
we obtain a homology cube as in \cite[Theorem~2.10 and Lemma~2.12]{CHM08}.
We see that $t_{(k,i),(l,j)}$ corresponds to the linking number of the loops $x_{k,h_k+i}$ and $x_{l,h_l+j}$ in the homology cube.
Thus, we obtain
\[
t_{(k,i),(l,j)}=t_{(l,j),(k,i)},
\]
namely, $T$ is symmetric.

On the other hand,
we can write
\[
\tilde{\rho}(A(P_J))-\tilde{\rho}(A(P))=
\begin{pmatrix}
O_{N_1,N_1}&O_{N_1,N_2}&O_{N_1,N_2}\\
O_{N_2,N_1}&O_{N_2,N_2}&W_J\\
O_{N_2,N_1}&O_{N_2,N_2}&X_J
\end{pmatrix}\in M(N,\hat{I}/\hat{I}^2),
\]
and it does not change under the above column transformations.
Here, $W_J=(\tilde{\rho}(W_{kl}))_{\substack{1\le k\le d\\1\le l\le d}}$  and $X_J=(\tilde{\rho}(X_{kl}))_{\substack{1\le k\le d\\1\le l\le d}}$,
where
\[
W_{kk}=\begin{pmatrix}
\overline{\dfrac{\partial c(x_{k,j})}{\partial x_{k,h_k+i}}}
\end{pmatrix}_{\substack{1\le i\le h_k\\1\le j\le h_k}}
\text{ and }
X_{kk}=\begin{pmatrix}
\overline{\dfrac{\partial c(x_{k,j})}{\partial x_{k,i}}}
\end{pmatrix}_{\substack{1\le i\le h_k\\1\le j\le h_k}}-I_{h_k}
\]
and $W_{kl}=X_{kl}=O_{h_k,h_l}$ if $k\ne l$.
Thus, we obtain
\begin{align*}
\ldet_1(\tilde{\rho}(A(P_J)))-\ldet_1(\tilde{\rho}(A(P)))
&=(\theta_*\circ \tr)((\tilde{\rho}(A(P_J))-\tilde{\rho}(A(P)))\epsilon_0^{-1})\\
&=(\theta_*\circ \tr)(-W_JT+X_J)\in H.
\end{align*}

First, we show that $\tr(W_JT)=0$.
Let $\pi'=\pi_1(\partial V_k\setminus \Int N(q_k))$,
and let $\ev_{x_{k,i}x_{k,j}}\colon H_1(\partial V_k;\Z)^{\otimes 3}\to H_1(\partial V_k;\Z)$ denote the map
\[
a\otimes b\otimes c\mapsto (a\cdot x_{k,i})(b\cdot x_{k,j})c
\]
for $1\le i\le 2h_k$ and $1\le j\le 2h_k$.
Fox's free derivative $\frac{\partial}{\partial x_{k,i}}\colon \pi'\to \Q\pi'$ induces a homomorphism
\[
\frac{\partial}{\partial x_{k,i}}\colon \pi'(2)\to \frac{I\pi'}{(I\pi')^2} \cong H_1(\partial V_k;\Z),
\]
where $I\pi'$ denotes the augmentation ideal of $\Q\pi'$.
We can check that
\[
\overline{\frac{\partial \varphi_k(x_{k,j})}{\partial x_{k,h_k+i}}}
=\overline{\frac{\partial(\varphi_k(x_{k,j})x_{k,j}^{-1})}{\partial x_{k,h_k+i}}}
=(\ev_{x_{k,j}x_{k,i}}\circ \tau_1)(\varphi_k) \in H_1(\partial V_k;\Z)
\]
for $1\le i\le h_k$ and $1\le j\le h_k$.
Since $\tau_1(\varphi_k)\in \Lambda^3(H_1(\partial V_k;\Z))$,
we see that 
\[
\overline{\frac{\partial \varphi_k(x_{k,j})}{\partial x_{k,h_k+i}}}
=(\ev_{x_{k,j}x_{k,i}}\circ \tau_1)(\varphi_k)
=-(\ev_{x_{k,i}x_{k,j}}\circ \tau_1)(\varphi_k)
=-\overline{\frac{\partial \varphi_k(x_{k,i})}{\partial x_{k,h_k+j}}}
\]
for $1\le i\le h_k$ and $1\le j\le h_k$.
Thus, $W_J\in M(N_2,\widehat{\Q\pi}/\hat{I}^2)$ is an alternating matrix,
and we see that $\tr(W_JT)=0$.

Next, we compute $(\theta_*\circ \tr)(X_J)$.
By the definition of $C_k$, we have
\[
\frac{1}{2}(C_k\circ\tau_1)(\varphi_k)
=\sum_{i=1}^{h_k}(\ev_{x_{k,i}x_{k,h_k+i}}\circ\tau_1)(\varphi_k)\in H_1(\partial V_k;\Z).
\]
On the other hand, we can check that
\[
\overline{\frac{\partial \varphi_k(x_{k,i})}{\partial x_{k,i}}}-1
=\overline{\frac{\partial(\varphi_k(x_{k,i})x_{k,i}^{-1})}{\partial x_{k,i}}}
=-(\ev_{x_{k,i}x_{k,h_k+i}}\circ \tau_1)(\varphi_k)\in H_1(\partial V_k;\Z)
\]
for $1\le i\le h_k$.
Thus, we obtain that
\begin{align*}
(\theta_*\circ \tr)(X_J)
&=(\incl_k)_*\left(\sum_{k\in J}\sum_{i=1}^{h_k}\overline{\frac{\partial(\varphi_k(x_{k,i})x_{k,i}^{-1})}{\partial x_{k,i}}}\right)\\
&=-\frac{1}{2}\sum_{k\in J}((\incl_k)_*\circ C_k\circ\tau_1)(\varphi_k)\in H,
\end{align*}
where we identify $H_1(M;\Z)\cong H_\Z$.
As a conclusion, we obtain
\[
\ldet_1(\tilde{\rho}(A(P_J)))-\ldet_1(\tilde{\rho}(A(P)))=-\frac{1}{2}\sum_{k\in J}((\incl_k)_*\circ C_k\circ\tau_1)(\varphi_k).
\]
Since
\[
\ldet_1\colon \Ker\left(K_1(\widehat{\Q\pi}/\hat{I}^2)\to K_1(\widehat{\Q\pi}/\hat{I})\right)\to H
\]
is an isomorphism as in Corollary~\ref{cor:gr}, we have
\[
\tilde{\rho}(A(P_J))
=\tilde{\rho}\left(-\frac{1}{2}\sum_{k\in J}((\incl_k)_*\circ C_k\circ\tau_1)(\varphi_k)\right) \tilde{\rho}(A(P))
\in K_1(\widehat{\Q\pi}/\hat{I}^2).
\]
\end{proof}

The next theorem describes the change of the Johnson homomorphism on $\I\C$ under Torelli surgery along $\Sigma_{h}$.
Let $\M_{h,1}$ denote the mapping class group of the surface $\Sigma_{h,1}$ which fixes the boundary pointwise.
We consider $\varphi\in\M_{h,1}$ as a mapping class of the surface $\Sigma_{h}$ by extending it by the identity map on the open disk $\Sigma_h\setminus \Sigma_{h,1}$.

\begin{theorem}[{\cite[Theorem~2]{GaLe05}}]\label{lem:johnsonhomo}
Let $(M,i) \in J_d\C$ and let $\iota\colon \Sigma_{h} \to M$ be an embedding.
For $\varphi \in J_d\M_{h,1}$, 
let $M_\varphi$ be the homology cylinder obtained from $M\setminus\Int(\iota(\Sigma_{h})\times[-1,1])$ and $\Sigma_{h}\times[-1,1]$ by the identification
\[
(\iota(x),1)=(\varphi(x),1)\text{ and }(\iota(x),-1)=(x,-1).
\]
Then, the homology cylinder $M_\varphi$ lies in $J_d\C$ and the equality
\[
\tau_d(M_\varphi) = \tau_d(M) +(i_-)_\ast^{-1}\circ\iota_\ast(\tau_d(\varphi)) \in H_\Z\otimes L_{d+1}(H_\Z)
\]
holds.
\end{theorem}

We can describe the Euler structure corresponding to $\tilde{\rho}(A(P_J))$.
Let 
\[
C\colon \Lambda^3 H_1(M) \xrightarrow{(i_+)_\ast^{-1}} \Lambda^3 H_1(\Sigma_{g,1})\to H_1(\Sigma_{g,1})
\]
be the contraction map defined in the same way as $C_k$.

\begin{proposition}
If $\tilde{\rho}(A(P))=\tilde{\alpha}(M)\in K_1(\widehat{\Q\pi})$, the equality
\[
\tilde{\rho}(A(P_J))
=\tilde{\rho}\left(\frac{1}{2}\sum_{k\in J}\Bigl(-((\incl_k)_*\circ C_k\circ\tau_1)(\varphi_k)+(C\circ(\incl_k)_*\circ \tau_1)(\varphi_k) \Bigr)\right)\tilde{\alpha}(M_J)
\]
in $K_1(\widehat{\Q\pi})$ holds.
In other words, if the presentation $P$ corresponds to the preferred Euler structure of $M$,
the presentation $P_J$ corresponds to the Euler structure
\[
\xi_0+\frac{1}{2}\sum_{k\in J}\Bigl(-((\incl_k)_*\circ C_k\circ\tau_1)(\varphi_k)+(C\circ(\incl_k)_*\circ \tau_1)(\varphi_k)\Bigr),
\]
where we also denote by $\xi_0$ the preferred Euler structure of $M_J$.
\end{proposition}
\begin{proof}
By Lemma~\ref{lem:euler-str},
\[
\tilde{\alpha}(M)
=\tau^{\tilde{\rho}}(M,\partial_-M,\xi_0)
=\tilde{\rho}\Bigl(-\frac{1}{2}C\circ\tau_1(M)\Bigr)\in K_1(\widehat{\Q\pi}/\hat{I}^2).
\]
By Lemma~\ref{lem:torsion-presentation},
we have
\[
\tilde{\rho}(A(P_J))=\tilde{\rho}\left(-\frac{1}{2}\sum_{k\in J}((\incl_k)_*\circ C_k\circ\tau_1)(\varphi_k)\right) \tilde{\rho}(A(P))
\in K_1(\widehat{\Q\pi}/\hat{I}^2).
\]
If $\tilde{\rho}(A(P))=\tilde{\rho}(-\frac{1}{2}C\circ\tau_1(M)) \in K_1(\widehat{\Q\pi}/\hat{I}^2)$,
we have
\begin{align*}
&\tilde{\rho}(A(P_J))\\
&=\tilde{\rho}\left(-\frac{1}{2}\sum_{k\in J}((\incl_k)_*\circ C_k\circ\tau_1)(\varphi_k)-\frac{1}{2}C\circ\tau_1(M)\right)\\
&=\tilde{\rho}\left(\frac{1}{2}\sum_{k\in J}\Bigl(-((\incl_k)_*\circ C_k\circ\tau_1)(\varphi_k)+(C\circ(\incl_k)_*\circ \tau_1)(\varphi_k)\Bigr)-\frac{1}{2}C\circ\tau_1(M_J)\right)\\
&=\tilde{\rho}\left(\frac{1}{2}\sum_{k\in J}\Bigl(-((\incl_k)_*\circ C_k\circ\tau_1)(\varphi_k)+(C\circ(\incl_k)_*\circ \tau_1)(\varphi_k)\Bigr)\right)\tilde{\alpha}(M_J)
\end{align*}
by Theorem~\ref{lem:johnsonhomo} 
as elements in $K_1(\widehat{\Q\pi}/\hat{I}^2)$,
and it shows what we desired.
\end{proof}

\begin{corollary}\label{cor:changeeulerstr}
For $d\ge2$,
\[
\prod_{J\subset\{1,2,\ldots, d\}}\tilde{\alpha}(M_J)^{(-1)^{|J|}}
=\prod_{J\subset\{1,2,\ldots, d\}}\tilde{\rho}(A(P_J))^{(-1)^{|J|}}
\in K_1(\widehat{\Q\pi}).
\]
\end{corollary}

\subsection{Proof of Theorem~\ref{thm:Y-filtration}}
When $d=1$, Theorem~\ref{thm:Y-filtration} is easy since
\[
\epsilon(\tilde{\alpha}(M))=\epsilon(\tilde{\alpha}(M_J))=1\in K_1(\Q).
\]

Assume $d\ge2$.
By Corollary~\ref{cor:changeeulerstr} and Proposition~\ref{prop:ldet},
it suffices to show that 
\[
\ldet\left(\prod_{J\subset\{1,2,\ldots, d\}}\tilde{\rho}(A(P_J))^{(-1)^{|J|}}\right) \in \Ker\left(H_1^{\Lie}(\hat{I})\to H_1^{\Lie}(\hat{I}/\hat{I}^d)\right).
\]

Define a map
\[
\Delta_d\colon M(N,\hat{I})\times M(N,\hat{I})\times \cdots \times M(N,\hat{I})\to H_1^\Lie(\hat{I})
\]
by $\Delta_d(A_1,\ldots, A_d)=\sum_{J\subset\{1,2,\ldots, d\}}(-1)^{|J|}\ldet(I_N+\sum_{k\in J} A_j)\in H_1^\Lie(\hat{I})$,
where $I_N$ denotes the identity matrix of size $N$.
\begin{lemma}\label{lem:altprod}
Let $A_1,A_2,\ldots,A_d\in M(N,\hat{I})$.
Then, we have
\[
\Delta_d(A_1,A_2,\ldots,A_d) \in \Ker\left(H_1^{\Lie}(\hat{I})\to H_1^{\Lie}(\hat{I}/\hat{I}^d)\right).
\]
\end{lemma}
\begin{proof}
\begin{align*}
\Delta_d(A_1,\ldots, A_d)
&=\sum_{J\subset\{1,2,\ldots, d\}}(-1)^{|J|}\ldet\left(I_N+\sum_{j\in J} A_j\right)\\
&= \sum_{k=1}^{d-1}\frac{(-1)^{k-1}}{k}\tr\left(\sum_{J\subset\{1,2,\ldots, d\}}(-1)^{|J|} \Biggl(\sum_{j\in J}A_j\Biggr)^k \right) \in H_1^{\Lie}(\hat{I}/\hat{I}^d).
\end{align*}
Here, note that the $k$th powers of matrices satisfy the finiteness property, that is, 
$\sum_{J\subset\{1,2,\ldots, d\}}(-1)^{|J|}(\sum_{j\in J}A_j)^k$ equals to the zero matrix for
$1\le k\le d-1$.
Thus, we obtain
$\Delta_d(A_1,A_2,\ldots,A_d) = 0\in H_1^{\Lie}(\hat{I}/\hat{I}^d)$.
\end{proof}
Let us denote $A_k=\tilde{\rho}(A(P_{\{k\}}))-\tilde{\rho}(A(P))$.
In the same way as the proof of Lemma~\ref{lem:torsion-presentation},
we have 
\[
\tilde{\rho}(A(P_J))=\tilde{\rho}(A(P))+\sum_{k\in J}A_k.
\]
Here, note that $A_k\in M(N,\hat{I})$ since $\varphi_k$ acts trivially on $H_1(\partial V_k;\Z)$.
We obtain
\begin{align*}
\ldet\left(\prod_{J\subset\{1,2,\ldots, d\}}\tilde{\rho}(A(P_J))^{(-1)^{|J|}}\right)
&=\sum_{J\subset\{1,2,\ldots, d\}}(-1)^{|J|}\ldet\left(\tilde{\rho}(A(P))+\sum_{k\in J}A_k\right)\\
&=\sum_{J\subset\{1,2,\ldots, d\}}(-1)^{|J|}\ldet\left(I_N+\sum_{k\in J}A'_k\right)\\
&=\Delta_d(A'_1,\ldots, A'_d),
\end{align*}
where we denote $A'_k=\tilde{\rho}(A(P))^{-1}A_k$.
By Lemma~\ref{lem:altprod},
it is in the kernel $\Ker\left(H_1^{\Lie}(\hat{I})\to H_1^{\Lie}(\hat{I}/\hat{I}^d)\right)$.
This finishes the proof of Theorem~\ref{thm:Y-filtration}.

%%%%%%%%%%%%%%%%%%
\section{A clasper surgery formula}
\label{section:change-rt}
In this section, we investigate the behavior of $\tilde{\alpha}(M)$ under surgeries along $k$-loop graph claspers for $k\ge1$.
We first describe the fundamental group of a tubular neighborhood of a graph clasper embedded in a 3-manifold.

%%%%%
\subsection{A lower bound on the degree over the fundamental group}
A graph clasper is an embedded surface in a $3$-manifold consisting of three kinds of constituents called leaves, nodes, and bands.
By assigning vertices to leaves and nodes and assigning edges to bands,
we obtain a uni-trivalent graph, which is similar to taking an inverse image in the surgery map in \cite[Section~2.2]{HaMa09}.
Note that we consider a uni-trivalent graph without labels in univalent vertices.

Let $G$ be a graph clasper of degree $d$,
and let $J$ be the corresponding uni-trivalent graph.
We write $E_U$ for the set of edges of $G$ incident to univalent vertices and $E_T$ for the set of half-edges of $G$ contained in edges which connect trivalent vertices, where a half-edge of $G$ is a piece obtained by cutting an edge of $G$ at its midpoint.
%We denote by $E_U$ and $E_T$ the set of edges incident to univalent vertices and the set of half-edges contained in edges each of which connects two trivalent vertices, respectively.
We also denote by $E(J)=E_U\cup E_T$.
The clasper $G$ can be decomposed into $Y$-graphs by \cite[Move~2]{Hab00C},
and there is a one-to-one correspondence between $E(J)$ and the set of leaves of the $Y$-graphs.

In this paper, we endow every leaf and node in $G$ with orientations.
An edge is said to be \emph{twisted} if the orientations of the nodes connected by the edge are not coherent.
Unless otherwise stated, the visible side of a node is positively oriented in the figures of this paper.
For $e\in E(J)$, we denote by $\alpha_e$ and $\beta_e$ a meridian and a longitude of the leaf, respectively (see Figure~\ref{fig:Y-graph}).
Note that their orientations depend on the orientation of the node incident to $e$.
Here, we care about only the conjugacy classes of $\alpha_e$ and $\beta_e \in \pi_1N(G)$, and do not specify the basepoint in $N(G)$.

For $x\in \pi_1(N(G))$,
we denote by 
\[
\deg(x)=\sup\{n\mid x\in \pi_1(N(G))(n)\}\in \Z_{\ge1}\cup\{\infty\}.
\]
Note that, if $\deg(x)=\infty$, we have $x=1\in\widehat{\Q\pi_1N(G)}$.
We use the fundamental (in)equalities
\begin{align*}
 \deg(xy) &\geq \min\{\deg x, \deg y\}, \\
 \deg(yxy^{-1})&=\deg x, \\
 \deg[x,y] &\geq \deg x+\deg y
\end{align*}
throughout this paper.
Here, we denote $[x,y]=xyx^{-1}y^{-1}$.

\begin{proposition}\label{prop:connected}
Let $e\in E(J)$.
\begin{enumerate}
 \item If $e\in E_U$,
we have $\deg\beta_e=1$.
 \item If $e\in E_T$ and $J\setminus\Int e$ is connected, we have $\deg\beta_e=\infty$.
\end{enumerate}
\end{proposition}

We treat the case when $e\in E_T$ and $J\setminus\Int e$ is disconnected in Proposition~\ref{prop:disconnected}.
Let $Y$ be a $Y$-graph, that is, a graph clasper of degree $1$.

\begin{figure}[h]
 \centering
 \includegraphics[width=0.35\textwidth]{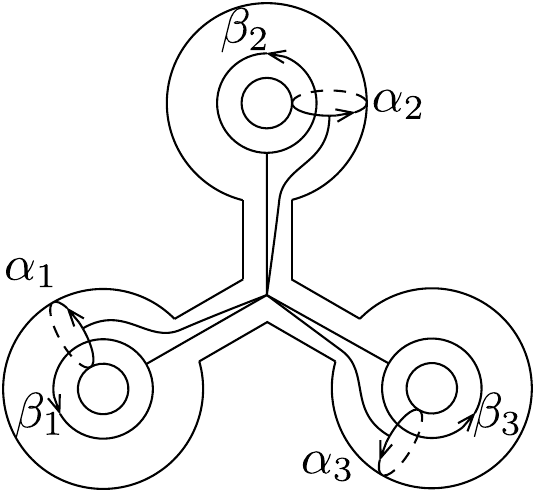}
 \caption{Meridians $\alpha_i$ and longitudes $\beta_i$.}
 \label{fig:Y-graph}
\end{figure}

\begin{lemma}
\label{lem:pi1(Y)}
The group $\pi_1(N(Y))$ has a presentation $P_Y$
with a generating set $\{\alpha_i,\beta_i\}_{i=1}^3$ as in Figure~\ref{fig:Y-graph} and relators
\begin{align*}
r_1&=\alpha_1\beta_3\alpha_2\beta_2^{-1}\alpha_2^{-1}\beta_3^{-1}\beta_2,\\
r_2&=\alpha_2\beta_1\alpha_3\beta_3^{-1}\alpha_3^{-1}\beta_1^{-1}\beta_3,\\
r_3&=\alpha_3\beta_2\alpha_1\beta_1^{-1}\alpha_1^{-1}\beta_2^{-1}\beta_1.
\end{align*}
\end{lemma}
\begin{proof}
Let us describe a handlebody of genus 3 as the complement of a neighborhood of a tangle consisting of 3 arcs in $D^3$.
The neighborhood $N(Y)$ is obtained by surgery along the link in the handlebody as in Figure~\ref{fig:Wirtinger}.
Denoting by the same symbol $a_{i,j}$ the meridian of the arc $a_{i,j}$ in Figure~\ref{fig:Wirtinger}, we obtain a presentation with generating set $\{a_{i,j}\}_{\substack{1\le i\le3\\1\le j\le 7}}$ and relators
\begin{figure}[h]
 \centering
 \includegraphics[width=0.7\textwidth]{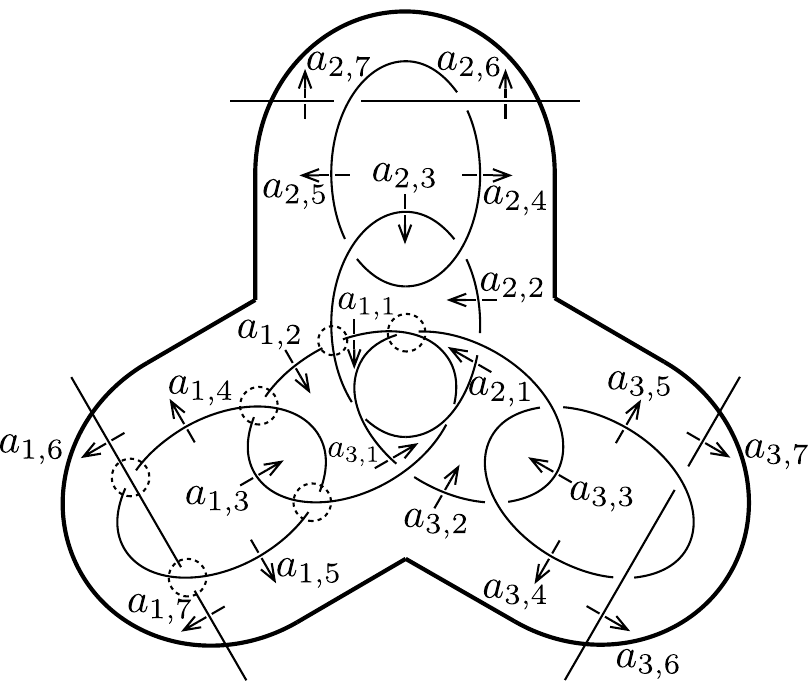}
 \caption{Generators $a_{i,j}$ of the complement of a clasper.}
 \label{fig:Wirtinger}
\end{figure}
\[
\begin{cases}
a_{i,1}a_{i+2,1}a_{i,1}^{-1}a_{i+2,3}^{-1},\\
a_{i,1}a_{i+1,3}^{-1}a_{i,2}^{-1}a_{i+1,3},\\
a_{i,2}a_{i,4}^{-1}a_{i,3}^{-1}a_{i,4},\\
a_{i,3}a_{i,4}a_{i,3}^{-1}a_{i,5}^{-1},\\
a_{i,5}a_{i,6}^{-1}a_{i,4}^{-1}a_{i,6},\\
a_{i,5}a_{i,7}a_{i,5}^{-1}a_{i,6}^{-1},
\end{cases}
\begin{cases}
a_{i,3}a_{i,6},\\
a_{i+1,3}^{-1}a_{i,4}^{-1}a_{i+1,1},
\end{cases}
\]
where the former relators come from the Wirtinger presentation of the tangle and the link,
and the latter comes from attaching $2$-handles by surgery.
For example, the former relators for $i=1$ correspond to the dotted circles in Figure~\ref{fig:Wirtinger}.
Using these relators, one can eliminate the generators $a_{i,j}$ for $i=1,2,3$ and $j=2,3,5,7$, and obtain a relation $a_{i,1} = a_{i+1,6}a_{i,4}^{-1}a_{i,6}^{-1}a_{i,4}a_{i+1,6}^{-1}$.
This relation eliminates $a_{i,1}$'s and gives a presentation with generating set $\{a_{i,4}, a_{i,6}\}_{1\le i\le 3}$ and relators $\{a_{i,4}^{-1}a_{i+2,6}a_{i+1,4}^{-1}a_{i+1,6}^{-1}a_{i+1,4}a_{i+2,6}^{-1}a_{i+1,6}\}_{1\le i\le 3}$.
By denoting $\alpha_{i}=a_{i,4}^{-1}$, $\beta_{i}=a_{i,6}$ as in Figure~\ref{fig:Y-graph},
the relators turn into $\alpha_i\beta_{i+2}\alpha_{i+1}\beta_{i+1}^{-1}\alpha_{i+1}^{-1}\beta_{i+2}^{-1}\beta_{i+1}$ for $1\le i\le 3$.
\end{proof}

\begin{lemma}\label{lem:superadditive}
Suppose that two trivalent vertices $v$ and $v'$ in $J$ is connected by an edge $e_1\cup e_1'$, where $e_1$ and $e_1'$ are half-edges incident to $v$ and $v'$, respectively.
We also denote by $e_2', e_3'\in E(J)$ the \textup{(}half-\textup{)}edges incident to $v'$ other than $e_1'$.
Then, for $\beta_{e_1}, \beta_{e_2'}, \beta_{e_3'} \in \pi_1N(G)$, we have
\[
\deg\beta_{e_1}\ge \deg\beta_{e_2'}+\deg\beta_{e_3'}.
\]
\end{lemma}

\begin{proof}
Let us denote by $Y$ and $Y'$ the $Y$-graphs corresponding to $v$ and $v'$, respectively.
We also denote simply $\alpha_i=\alpha_{e_i}$, $\beta_i=\beta_{e_i}$, $\alpha'_i=\alpha_{e_i'}$, and $\beta_i'=\beta_{e_i'}$ for $1\le i\le 3$.
If the edge $e\cup e'$ is untwisted,
by van Kampen's theorem, we have the relation
\[
l\alpha_1'\bar{l}=\beta_1^{-1},\quad
l\beta_1'\bar{l}=\alpha_1^{-1}
\in \pi_1(N(G),\tilde{*}),
\]
where $l$ is a path in $\partial (N(G))$ from the basepoint $\tilde{*}$ in the neighborhood $N(Y)$ to $\tilde{*'}$ in $N(Y')$ (see Figure~\ref{fig:path-l}).
\begin{figure}[h]
 \centering
 \includegraphics[width=0.5\textwidth]{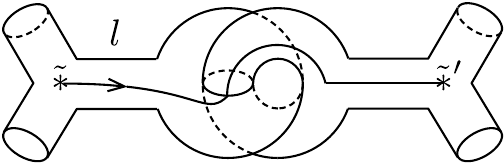}
 \caption{A path $l$ from $\tilde{\ast}$ to $\tilde{\ast}'$.}
 \label{fig:path-l}
\end{figure}
If the edge $e\cup e_1'$ is half-twisted, we may assume that the leaf corresponding to $e$ is half-twisted.
Since the disk twist induces the identity on $\pi_1(N(G))$,
we may assume that the leaf is positively half-twisted.
We consider that the leaf $e_1$ depicted on the left of Figure~\ref{fig:path-l} is also positively half-twisted,
and we define the path $l$ by the same figure.
In this case, we also have the relations
\[
l\alpha'_1\bar{l}=\beta_1^{-1}\alpha_1\beta_1\alpha_1^{-1}\beta_1, \quad
l\beta'_1\bar{l}=\beta_1^{-1}\alpha_1\beta_1.
\]
We may assume that the edges $e_1', e_2', e_3'$ are ordered clockwise as in Figure~\ref{fig:Y-graph}.
Then one has
\[
\alpha_1'=[\beta_2'^{-1},\beta'_3]\beta_3'[\beta_2'^{-1},\alpha_2']\beta_3'^{-1},\quad
\alpha_2'=[\beta_3'^{-1},\beta'_1]\beta_1'[\beta_3'^{-1},\alpha_3']\beta_1'^{-1}
\]
by Lemma~\ref{lem:pi1(Y)}.
Hence we obtain
\begin{align*}
 \deg\alpha_2' &\geq \min\{\deg[\beta_3'^{-1},\beta'_1], \deg[\beta_3'^{-1},\alpha_3']\} \\
 &\geq \min\{\deg\beta_3'+\deg\beta'_1, \deg\beta_3'+\deg\alpha_3'\} \\
 &> \deg\beta_3'.
\end{align*}
Therefore,
\begin{align*}
 \deg\beta_1 = \deg\alpha_1' 
 &\geq \min\{\deg[\beta_2'^{-1},\beta'_3], \deg[\beta_2'^{-1},\alpha_2']\} \\
 &= \deg\beta_2'+\deg\beta_3'.
\end{align*}
\end{proof}

\begin{proof}[Proof of Proposition~\ref{prop:connected}]
(1)
First note that
\[
\frac{\pi_1N(G)}{\pi_1N(G)(2)} = H_1(N(G);\Z)
\]
is a free abelian group of rank $\#E_U + b_1(J)$, where $b_1(J)$ denotes the first Betti number of $J$.
If $e \in E_U$, then $\beta_e$ is non-trivial in $H_1(N(G);\Z)$, and hence $\deg\beta_e=1$.

(2)
Since $J\setminus\Int e$ is connected, there is a cycle $C$ in $J$ containing the half-edge $e$.
We denote by $k$ the number of edges of $C$.
Applying Lemma~\ref{lem:superadditive} along $C$ inductively, we obtain $\deg\beta_e \geq \deg\beta_e+k$.
Thus, we conclude that $\deg\beta_e = \infty$.
\end{proof}

Using an isotopy, we have an embedding $\iota_0\colon G \hookrightarrow N_0(G)$ such that its image is disjoint from $G \subset N_0(G)$.
This map derives an embedding $\iota\colon G \hookrightarrow N(G)$.
Since $\iota_0$ induces an isomorphism on $H_1(\textendash;\Z)$ and a surjection on $H_2(\textendash;\Z)$, the same holds for $\iota$ by \cite[Proposition~5.4]{HaMa12} for instance.
It follows from \cite[Theorem~5.1]{Sta65} that $\iota$ induces an isomorphism
\[
\frac{\pi_1(G)}{\pi_1(G)(n)}\cong \frac{\pi_1N(G)}{\pi_1N(G)(n)}
\]
between the nilpotent quotients for $n\ge1$.
Furthermore, $\iota$ gives an isomorphism $\frac{\pi_1N(G)(n)}{\pi_1N(G)(n+1)} \cong L_n(H_1(G;\Z))$, where $L_n(H_1(G;\Z))$ denotes the degree $n$ part of the free Lie algebra generated by a module $H_1(G;\Z)$.

\begin{proposition}\label{prop:disconnected}
Suppose that $J\setminus\Int e$ is disconnected for a half-edge $e\in E_T$.
Let us denote by $J_e$ a connected component of $J\setminus \Int e$ which does not contain the trivalent vertex incident to $e$ in $J$.
\begin{enumerate}
 \item If $J_e$ is a tree, we regard the endpoint in $e$ as the root of $J_e$ and assign $[\beta_{e'}] \in H_1(N(G);\Z)$ to each leaf $e'\in E_U$ of $J_e$.
Giving the cyclic order coming from the orientation of nodes in $G$ to the trivalent vertices in $J_e$,
we regard $J_e$ as a Jacobi diagram.
Let $n\geq 2$ be the number of leaves of $J_e$,
and let $T_e \in L_n(H_1(G;\Z))$ denote the iterated Lie bracket encoded by $J_e$,
which is denoted as $\comm(J_e)$ in \cite[Section~4.3]{HaMa12}.
Then, $\beta_e=(-1)^k T_e$ holds under the above isomorphism, where $k$ is the number of twists in $J_e$. 
 \item If $J_e$ is not a tree, $\deg \beta_e=\infty$ holds.
\end{enumerate}
\end{proposition}

\begin{proof}
(1)
First note that $\deg\beta_e \geq n$ by Lemma~\ref{lem:superadditive}, namely $\beta_e \in \pi_1N(G)(n)$.
We use the same notation as in Lemma~\ref{lem:superadditive} by putting $e_1=e$.
Let us prove by induction on $n$.
When $n=2$, the graph $J_e$ is a $Y$-graph.
It follows from the equalities in the proof of Lemma~\ref{lem:superadditive} that
\[
\beta_e^{-1} = l\alpha_1'\bar{l} = l[\beta_2'^{-1},\beta'_3]\beta_3'[\beta_2'^{-1},\alpha_2']\beta_3'^{-1}\bar{l}
\]
when $e_2'$, $e_3'$, and the edge containing $e$ are untwisted.
If the edge is twisted, the leftmost $\beta_e^{-1}$ should be replaced with $\beta_1^{-1}\alpha_1\beta_e\alpha_1^{-1}\beta_1$.
Also, the sign of the exponent of $\beta_j'$ in the right-hand side should be switched if $e_j'$ is twisted for $j=2,3$.
Since $\deg\alpha_2' > \deg\beta'_3$, we have 
\[
\beta_e=(-1)^k[\beta_2',\beta'_3]=(-1)^kT_e \in \frac{\pi_1N(G)(2)}{\pi_1N(G)(3)}.
\]

Suppose $n>2$.
The induction hypothesis implies 
\[
[\beta_2'^{-1},\beta'_3] = (-1)^{k'}[-T_{e_2'},T_{e_3'}] \in \frac{\pi_1N(G)(n)}{\pi_1N(G)(n+1)},
\]
where $k'=k$ if the edge containing $e$ is untwisted, and $k'=k-1$ if it is twisted.
By the same argument as in the case $n=2$, we conclude that $\beta_e = (-1)^k[T_{e_2'},T_{e_3'}] = (-1)^kT_e$.

(2)
Since $J_e$ is not a tree, there is a cycle in $J_e$.
Take a path from a trivalent vertex in the cycle to $e$.
By Proposition~\ref{prop:connected}(2), $\deg \beta_e'=\infty$ for every half-edge in $C$.
Applying Lemma~\ref{lem:superadditive} along the path inductively, we obtain $\deg\beta_e = \infty$.
\end{proof}

%%%
\subsection{The invariance of $\tau$ under $k$-loop surgeries for $k\ge2$}
Here, we prove the following:
\begin{theorem}\label{thm:k-loop}
Let $k\ge2$,
and let $G$ be a $k$-loop graph clasper in a homology cylinder $M$.
We have
\[
\tilde{\alpha}(M_G)=\tilde{\alpha}(M).
\]
\end{theorem}

Theorem~\ref{thm:k-loop} follows from the two lemmas below.
\begin{lemma}~\label{lem:Y-graph}
Let $Y$ be a $Y$-graph in a homology cylinder $M$ such that the longitudes of the three leaves represent $1\in\widehat{\Q\pi_1(M_Y)}$.
We have
\[
\tilde{\alpha}(M_Y)=\tilde{\alpha}(M).
\]
\end{lemma}

\begin{lemma}\label{lem:trivvert-infty}
Let $G$ be a graph clasper in a homology cylinder $M$,
and let $J$ be the corresponding Jacobi diagram.
If $G$ is a $k$-looped graph clasper for $k\ge2$,
there is a trivalent vertex in $J$ such that $\deg\beta_e=\infty$ for the three half-edges $e$ incident to the vertex.
\end{lemma}

\begin{proof}
Since $b_1(J)\geq 2$, there are two distinct cycles $C_1$ and $C_2$ in $J$.
In the case $C_1\cap C_2 \neq \emptyset$, there exists a trivalent vertex $v \in C_1\cap C_2$.
Proposition~\ref{prop:connected}(2) implies that $v$ is what we want.

Next we consider the case $C_1\cap C_2 = \emptyset$.
Then there exists a path $P$ connecting $C_1$ and $C_2$ such that $P\cap C_j$ is a vertex, say $v_j$, for $j=1,2$.
It follows from Proposition~\ref{prop:connected}(2) that $\deg\beta_e = \infty$ for the two half-edges $e \subset C_j$ incident to $v_j$.
Applying Lemma~\ref{lem:superadditive} along $P$ inductively, we also conclude that $\deg\beta_e = \infty$ for the half-edge $e\subset P$ incident to $v_2$.
\end{proof}

For a 3-manifold $M$ and a presentation $P$ of $\pi_1M$ with generators $\{x_j\}_{j=1}^p$ and relators $\{r_i\}_{i=1}^q$,
we denote by $A(P)$ the $p\times q$ matrix whose $(i,j)$-entry is $\overline{\frac{\partial r_i}{\partial x_j}}\in\widehat{\Q\pi_1(M)}$.
Here, for $v\in \Z\pi_1M$, $\bar{v}\in \Z\pi_1M$ is defined in Section~\ref{section:compute-RT}.
First, we prove Lemma~\ref{lem:Y-graph}.
Let $N(Y)$ denote a tubular neighborhood of $Y$,
and let $\{\alpha_i,\beta_i\}_{i=1}^3$ be a generating set of $\pi_1(N(Y))$ depicted in Figure~\ref{fig:Y-graph}.

\begin{proof}[Proof of Lemma~\ref{lem:Y-graph}]
Since 
\[
\alpha_i=[\beta_{i+1}^{-1},\beta_{i+2}]\beta_{i+2}[\beta_{i+1}^{-1},\alpha_{i+1}]\beta_{i+2}^{-1}\in\pi_1(N(Y))
\]
for $i=1,2,3$,
we see that $\alpha_i=1\in\widehat{\Q\pi_1(M_G)}$ if $\beta_{i+1}=1\in\widehat{\Q\pi_1(M_G)}$.
With respect to the presentation $P_Y$ given in Lemma~\ref{lem:pi1(Y)},
\[
A(P_Y)=
\begin{pmatrix}
\left(\overline{\dfrac{\partial r_l}{\partial \alpha_k}}\right)_{\substack{1\le k\le 3\\1\le l\le 3}}\\
\left(\overline{\dfrac{\partial r_l}{\partial \beta_k}}\right)_{\substack{1\le k\le 3\\1\le l\le 3}}
\end{pmatrix},
\]
where
\begin{align*}
\left(\frac{\partial r_l}{\partial \alpha_k}\right)_{\substack{1\le k\le 3\\1\le l\le 3}}
&=\begin{pmatrix}
1&0&(\alpha_3-\beta_1^{-1})\beta_2\\
(\alpha_1-\beta_2^{-1})\beta_3&1&0\\
0&(\alpha_2-\beta_3^{-1})\beta_1&1
\end{pmatrix},\\
\left(\frac{\partial r_l}{\partial \beta_k}\right)_{\substack{1\le k\le 3\\1\le l\le 3}}
&=\begin{pmatrix}
0&\alpha_2-\beta_3^{-1}&\beta_1^{-1}(1-\beta_2\alpha_1)\\
\beta_2^{-1}(1-\beta_3\alpha_2)&0&\alpha_3-\beta_1^{-1}\\
\alpha_1-\beta_2^{-1}&\beta_3^{-1}(1-\beta_1\alpha_3)&0
\end{pmatrix}.
\end{align*}
It is the same as that of the handlebody of genus three.
Thus, If we replace the neighborhood of $G$ with the handlebody,
the torsion does not change,
and we obtain
\[
\tilde{\alpha}(M_G)=\tilde{\alpha}(M).
\]
\end{proof}

%%%
\subsection{Values on $1$-loop surgeries}
Here, we prove the following surgery formula of $\tilde{\alpha}\colon \I\C\to K_1(\widehat{\Q\pi})$ for 1-loop graph claspers.

\begin{theorem}\label{thm:value-1-loop}
Let $G$ be a $1$-loop graph clasper of degree $d$ embedded in a homology cylinder $M$ as in Figure~\ref{fig:MaMe13},
where $\epsilon_i$ denotes a rectangle or a rectangle with a positive half-twist.
We choose a path $l$ in $M$ from a basepoint $*\in m_-(\Sigma_{g,1})$ to $*_1$,
and let $\gamma_i=l f_i\bar{l}$ and $\delta=l f_0 \bar{l}$.
We also denote $\epsilon_i=1,0\in\Z/2\Z$, which is the same symbol as the rectangle,
depending on whether the rectangle $\epsilon_i$ is twisted or not.
Then, we have the equality
\[
\tilde{\alpha}(M_G)=
\left(\delta+(-1)^{\epsilon+1}\prod_{i=1}^d(1-\gamma_{d+1-i})\right)
\left(\delta^{-1}+(-1)^{\epsilon+1}\prod_{i=1}^d(1-\gamma_i^{-1})\right)
\tilde{\alpha}(M)
\]
in $(\widehat{\Q\pi})_{\ab}^{\times}$, where we identify $\widehat{\Q\pi_1(M)}\cong \widehat{\Q\pi}$
and denote $\epsilon=\sum_{i=1}^d\epsilon_i\in\mathbb{Z}/2\mathbb{Z}$.
\end{theorem}

\begin{figure}[h]
 \centering
 \includegraphics[width=0.8\textwidth]{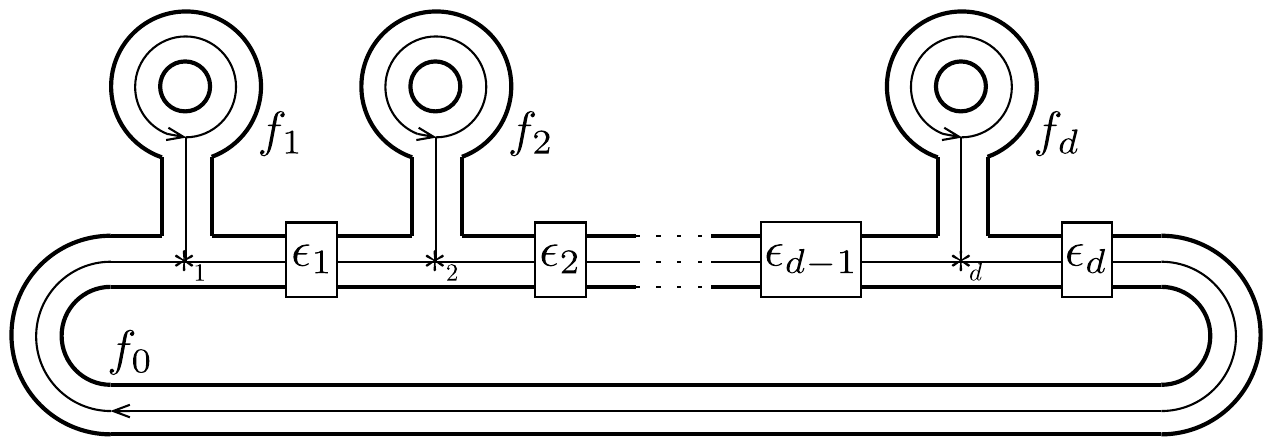}
 \caption{A $1$-loop graph clasper with twists $\epsilon_j$, basepoints $\ast_j$, and loops $f_j$.}
 \label{fig:MaMe13}
\end{figure}

\begin{remark}
A surgery formula of the $K_1(Q(\Z H_\Z))$-valued torsion $\alpha$ of homology cylinders for 1-loop surgeries is given in \cite[Lemma~3.15]{MaMe13}.
\end{remark}

\begin{proof}[Proof of Theorem~\ref{thm:value-1-loop}]
Let us decompose $G$ into $Y$-graphs embedded in $M$ by splitting the edges in $G$ each of which connects two nodes.
We denote by $Y_i$ the $Y$-graph containing the leaf whose one longitude represents the same free loop as $f_i$,
and assign a positive half-twist to one leaf of $Y_i$ if $\epsilon_i$ is half-twisted for $1\le i\le d$.
The regular neighborhood $N_0(Y_i)$ can be considered as a handlebody of genus three,
and we choose it so that $N_0(G)=\bigcup_{i=1}^d N_0(Y_i)$ as in Figure~\ref{fig:path-l}.
Let us also denote by $N(G)$ the corresponding neighborhood of $G$ in $M_G$,
and we identify $M \setminus \Int N_0(G)$ and $M_G \setminus \Int N(G)$.

Let us denote by $\tilde{*}_i$ points in $\partial N_0(G)$ obtained by pushing the points $*_i\in Y_i$ depicted in Figure~\ref{fig:MaMe13} to an identical normal direction.
We also pick a path $l_i$ in $\partial N_0(G)$ connecting $\tilde{*}_i$ and $\tilde{*}_{i+1}$ in the same way as $l$ in Figure~\ref{fig:path-l}.
Here, we denote $\tilde{*}_{d+1}=\tilde{*}_1$.

We consider that the leaves $\beta_1$, $\beta_2$, and $\beta_3$ in Figure~\ref{fig:Y-graph} correspond to
the leaf of $Y_i$ whose longitude represents the same free loop as $f_i$,
the half-edge connecting to $Y_{i+1}$,
and the half-edge connecting to $Y_{i-1}$, respectively.
For $1\le i\le d$ and $1\le j\le 3$,
let us denote by $\alpha_{i,j}\in\pi_1(\partial N_0(G),\tilde{*}_1)$ (resp.\ $\beta_{i,j}$) the loop obtained by connecting $l_1l_2\cdots l_{i-1}$ to the meridian corresponding to $\alpha_j$ (resp.\ the longitude corresponding to $\beta_j$) in $\partial N_0(Y_i)$.
We also denote the loop $\tilde{l}=l_1l_2\cdots l_d$.
By applying van Kampen's theorem when we glue $Y$-graphs,
we obtain a presentation $P_N$ of $\pi_1N(G)$ with the generating set
\[
\{\alpha_{i,j},\beta_{i,j}\mid 1\le i\le d, 1\le j\le 3\}\cup\{\tilde{l}\}
\]
and relators
\begin{align*}
r_{i,1}&=\alpha_{i,1}\beta_{i,3}\alpha_{i,2}\beta_{i,2}^{-1}\alpha_{i,2}^{-1}\beta_{i,3}^{-1}\beta_{i,2},\\
r_{i,2}&=\alpha_{i,2}\beta_{i,1}\alpha_{i,3}\beta_{i,3}^{-1}\alpha_{i,3}^{-1}\beta_{i,1}^{-1}\beta_{i,3},\\
r_{i,3}&=\alpha_{i,3}\beta_{i,2}\alpha_{i,1}\beta_{i,1}^{-1}\alpha_{i,1}^{-1}\beta_{i,2}^{-1}\beta_{i,1},\\
r_{i,4}&=
\begin{cases}
\alpha_{i,2}\beta_{i+1,3}
&\text{if $\epsilon_i=0$ and $i<d$,}\\
\beta_{i,2}^{-1}\alpha_{i,2}\beta_{i,2}\beta_{i+1,3}^{-1}
&\text{if $\epsilon_i=1$ and $i<d$,}\\
\alpha_{d,2}\tilde{l}\beta_{1,3}\tilde{l}^{-1}
&\text{if $\epsilon_i=0$ and $i=d$,}\\
\beta_{d,2}^{-1}\alpha_{d,2}\beta_{d,2}\tilde{l}\beta_{1,3}^{-1}\tilde{l}^{-1}
&\text{if $\epsilon_i=1$ and $i=d$,}\\
\end{cases}\\
r_{i,5}&=
\begin{cases}
\beta_{i,2}\alpha_{i+1,3}
&\text{if $\epsilon_i=0$ and $i<d$,}\\
\beta_{i,2}^{-1}\alpha_{i,2}\beta_{i,2}\alpha_{i,2}^{-1}\beta_{i,2}\alpha_{i+1,3}^{-1}
&\text{if $\epsilon_i=1$ and $i<d$,}\\
\beta_{d,2}\tilde{l}\alpha_{1,3}\tilde{l}^{-1}
&\text{if $\epsilon_i=0$ and $i=d$,}\\
\beta_{d,2}^{-1}\alpha_{d,2}\beta_{d,2}\alpha_{d,2}^{-1}\beta_{d,2}\tilde{l}\alpha_{1,3}^{-1}\tilde{l}^{-1}
&\text{if $\epsilon_i=1$ and $i=d$.}
\end{cases}
\end{align*}

By slightly pushing the endpoint of $l$ attached to $*_1$ in the normal direction of $G$,
we have a path connecting the basepoint $*\in m_-(\Sigma_{g,1})$ to $\tilde{*}_1$.
We may assume $\tilde{l}$ is contained in $M\setminus G$.
Changing the basepoint of loops from $\tilde{*}_1$ to $*$ via $\tilde{l}$,
the loops $\beta_{i,3}$ and $\tilde{l}$ are homotopic to $\gamma_i$ and $\delta$ in $M$, respectively.
Under the isomorphism
\[
\widehat{\Q\pi_1(M,*)}\cong \widehat{\Q\pi_1(M_G,*)}\cong \widehat{\Q\pi}
\]
coming from \cite[Theorem~5.1]{Sta65} and \cite{Qui68},
we have
\[
\alpha_{i,1}=\alpha_{i,2}=\alpha_{i,3}=\beta_{i,2}=\beta_{i,3}=1,\quad
\beta_{i,1}=\gamma_i, \quad
\tilde{l}=\delta\in\widehat{\Q\pi}
\]
by Proposition~\ref{prop:disconnected}(2) and the relators $r_{i,1}$, $r_{i,2}$, and $r_{i,3}$.

Denote the generators of $\pi_1(M_G,*)$ as
\[
x_{6(i-1)+2j-1}=\alpha_{i,j},\quad
x_{6(i-1)+2j}=\beta_{i,j},\quad
x_{6d+1}=\tilde{l}
\]
for $1\le i\le d$ and $1\le j\le 3$.

Let $A_N$ be a $(6d+1)\times 5d$ matrix defined by
\[
A_N=
\begin{pmatrix}
\biggl(
\overline{\dfrac{\partial r_{l,2}}{\partial x_k}}
\biggr)_{k,l}
&
\biggl(
\overline{\dfrac{\partial r_{l,4}}{\partial x_k}}
\biggr)_{k,l}
&
\biggl(
\overline{\dfrac{\partial r_{l,3}}{\partial x_k}}
\biggr)_{k,l}
&
\biggl(
\overline{\dfrac{\partial r_{l,1}}{\partial x_k}}
\biggr)_{k,l}
&
\biggl(
\overline{\dfrac{\partial r_{l,5}}{\partial x_k}}
\biggr)_{k,l}
\end{pmatrix},
\]
where $1\leq k\leq 6d+1$ and $1\leq l\leq d$.
In the following, we consider the matrix $A_N$ when $d\ge2$.
Theorem~\ref{thm:value-1-loop} for the case when $d=1$ is proved in the same way.
We have
\begin{align*}
\biggl(
\overline{\dfrac{\partial r_{l,2}}{\partial x_k}}
\biggr)_{\substack{1\le k\le 6d+1\\1\le l\le d}}
&=
\begin{pmatrix}
v_1& \bm{0}&\cdots&\bm{0}&\bm{0}\\
\bm{0}&v_2&\cdots&\bm{0}&\bm{0}\\
\vdots&\vdots&&\vdots&\vdots\\
\bm{0}&\bm{0}&\cdots&v_{d-1}&\bm{0}\\
\bm{0}&\bm{0}&\cdots&\bm{0}&v_d\\
0&0&\cdots&0&0
\end{pmatrix},
\\
\biggl(
\overline{\dfrac{\partial r_{l,4}}{\partial x_k}}
\biggr)_{\substack{1\le k\le 6d+1\\1\le l\le d}}
&=
\begin{pmatrix}
e_3& \bm{0}&\cdots&\bm{0}&(-1)^{\epsilon_d}\delta^{-1}e_6\\
(-1)^{\epsilon_1}e_6&e_3&\cdots&\bm{0}&\bm{0}\\
\bm{0}&(-1)^{\epsilon_2}e_6&\cdots&\bm{0}&\bm{0}\\
\vdots&\vdots&&\vdots&\vdots\\
\bm{0}&\bm{0}&\cdots&e_3&\bm{0}\\
\bm{0}&\bm{0}&\cdots&(-1)^{\epsilon_{d-1}}e_6&e_3\\
0&0&\cdots&0&0
\end{pmatrix},
\end{align*}
where $v_i=
\begin{pmatrix}
0&0&1&0&0&1-\gamma_i^{-1}
\end{pmatrix}^T$,
$\bm{0}=
\begin{pmatrix}
0&0&0&0&0&0
\end{pmatrix}^T$,
and $e_i\in\mathbb{R}^6$ is the $i$th standard unit vector for $1\le i\le 6$.

Applying column addition transformations to the submatrix
$L=\begin{pmatrix}
\biggl(
\overline{\dfrac{\partial r_{l,2}}{\partial x_k}}
\biggr)_{k,l}
&
\biggl(
\overline{\dfrac{\partial r_{l,4}}{\partial x_k}}
\biggr)_{k,l}
\end{pmatrix}$
of $A_N$, one can change
$\biggl(
\overline{\dfrac{\partial r_{l,2}}{\partial x_k}}
\biggr)_{k,l}$
into the form
\[
\begin{pmatrix}
e_3& \bm{0}&\cdots&\bm{0}&\bm{0}\\
\bm{0}&e_3&\cdots&\bm{0}&\bm{0}\\
\vdots&\vdots&&\vdots&\vdots\\
\bm{0}&\bm{0}&\cdots&e_3&\bm{0}\\
\bm{0}&\bm{0}&\cdots&\bm{0}&v_d'\\
0&0&\cdots&0&0
\end{pmatrix},
\]
where 
$v_d'=
\begin{pmatrix}
0&0&1+(-1)^{\epsilon+1}\delta\prod_{i=1}^d(1-\gamma_i^{-1})&0&0&0
\end{pmatrix}^T$.
Indeed, we first add $-(-1)^{\epsilon_{d-1}}(1-\gamma_d^{-1})$ times the $(2d-1)$st column and $(-1)^{\epsilon_{d-1}}(1-\gamma_d^{-1})$ times the $(d-1)$st column of $L$ to the $d$th column.
Here, scalar multiplication is done from the right.
Then, the $(6d-6,d)$-entry is $(-1)^{\epsilon_{d-1}}(1-\gamma_{d-1}^{-1})(1-\gamma_d^{-1})$.
We next continue this process and see that the $d$th column is $
\begin{pmatrix}
 \bm{0} & \cdots & \bm{0} & v'_d & 0
\end{pmatrix}^T
$.
Since $1+(-1)^{\epsilon+1}\delta\prod_{i=1}^d(1-\gamma_i^{-1}) \in \widehat{\Q\pi}$ is a unit, a similar procedure changes the $(d-1)$st column into a standard unit vector.
Continuing this process, we finally obtain the desired matrix.

We also have
\begin{align*}
\biggl(
\overline{\dfrac{\partial r_{l,1}}{\partial x_k}}
\biggr)_{\substack{1\le k\le 6d+1\\ 1\le l\le d}}
&=
\begin{pmatrix}
e_1& \bm{0}&\cdots&\bm{0}&\bm{0}\\
\bm{0}&e_1&\cdots&\bm{0}&\bm{0}\\
\vdots&\vdots&&\vdots&\vdots\\
\bm{0}&\bm{0}&\cdots&e_1&\bm{0}\\
\bm{0}&\bm{0}&\cdots&\bm{0}&e_1\\
0&0&\cdots&0&0
\end{pmatrix},
\\
\biggl(
\overline{\dfrac{\partial r_{l,5}}{\partial x_k}}
\biggr)_{\substack{1\le k\le 6d+1\\1\le l\le d}}
&=
\begin{pmatrix}
e_4& \bm{0}&\cdots&\bm{0}&(-1)^{\epsilon_d}\delta^{-1}e_5\\
(-1)^{\epsilon_1}e_5&e_4&\cdots&\bm{0}&\bm{0}\\
\bm{0}&(-1)^{\epsilon_2}e_5&\cdots&\bm{0}&\bm{0}\\
\vdots&\vdots&&\vdots&\vdots\\
\bm{0}&\bm{0}&\cdots&e_4&\bm{0}\\
\bm{0}&\bm{0}&\cdots&(-1)^{\epsilon_{d-1}}e_5&e_4\\
0&0&\cdots&0&0
\end{pmatrix}.
\end{align*}
Let $w_i=
\begin{pmatrix}
1-\gamma_i&0&0&1-\gamma_i&1&0
\end{pmatrix}^T$.
Applying column addition transformations to $A_N$,
one can also change the submatrix
\[
\begin{pmatrix}
\overline{\dfrac{\partial r_{l,3}}{\partial x_k}}
\end{pmatrix}_{\substack{1\le k\le 6d+1\\ 1\le l\le d}}
=\begin{pmatrix}
w_1& \bm{0}&\cdots&\bm{0}&\bm{0}\\
\bm{0}&w_2&\cdots&\bm{0}&\bm{0}\\
\vdots&\vdots&&\vdots&\vdots\\
\bm{0}&\bm{0}&\cdots&w_{d-1}&\bm{0}\\
\bm{0}&\bm{0}&\cdots&\bm{0}&w_d\\
0&0&\cdots&0&0
\end{pmatrix}
\]
into the form
\[
\begin{pmatrix}
w_1'& \bm{0}&\cdots&\bm{0}&\bm{0}\\
\bm{0}&e_5&\cdots&\bm{0}&\bm{0}\\
\vdots&\vdots&&\vdots&\vdots\\
\bm{0}&\bm{0}&\cdots&e_5&\bm{0}\\
\bm{0}&\bm{0}&\cdots&\bm{0}&e_5\\
0&0&\cdots&0&0
\end{pmatrix}
\]
in the same way,
where 
$w_1'=
\begin{pmatrix}
0&0&0&0&1+(-1)^{\epsilon+1}\delta^{-1}\prod_{i=1}^{d}(1-\gamma_{d+1-i})&0
\end{pmatrix}^T$.

The presentation $P_G$ of the fundamental group of $M_G=(M\setminus\Int N_0(G))\cup N(G)$ is obtained by adding a generating set of $\pi_1(M\setminus\Int N_0(G))$ and some relators to $P_N$.
On the other hand, the presentation $P$ of $\pi_1M$ is obtained by adding the same generating set and relators of $\pi_1(M\setminus\Int N_0(G))$ to $P_N$ and changing the relators $r_{i,1}$, $r_{i,2}$, $r_{i,3}$ into
\[
r_{i,1}=\alpha_{i,1},\ 
r_{i,2}=\alpha_{i,2},\ 
r_{i,3}=\alpha_{i,3}.
\]
The difference between the presentations of $\pi_1(M_G)$ and $\pi_1(M)$ only lies in $r_{i,1}$, $r_{i,2}$, and $r_{i,3}$.
Thus, the difference between $A(P_G)$ and $A(P)$ is concentrated in $v_1'$ and $w_1'$,
and we obtain the following equalities in $K_1(\widehat{\Q\pi})/{\pm}H_\Z$:
\begin{align*}
\tilde{\alpha}(M_G)
&=
\left(1+(-1)^{\epsilon+1}\delta\prod_{i=1}^d(1-\gamma_i^{-1})\right)
\left(1+(-1)^{\epsilon+1}\delta^{-1}\prod_{i=1}^{d}(1-\gamma_{d+1-i})\right)\tilde{\alpha}(M)\\
&=\left(\delta^{-1}+(-1)^{\epsilon+1}\prod_{i=1}^d(1-\gamma_i^{-1})\right)\left(\delta+(-1)^{\epsilon+1}\prod_{i=1}^d(1-\gamma_{d+1-i})\right)\tilde{\alpha}(M).
\end{align*}
Since $\tau_1(M)=\tau_1(M_G)$,
we have the same equality in $K_1(\widehat{\Q\pi})$.
\end{proof}

Recall the surgery map defined in \cite[Section~8.5]{Hab00C} and \cite[Section~2.2]{HaMa09}.
Let $\J$ denote the set of connected Jacobi diagrams with at least one trivalent vertex and with univalent vertices colored by $H_\Z$ and totally ordered.
The module $\A^{<,c}(H_\Z)$ is obtained by taking the quotient of free module $\Z\J$ generated by $\J$ with respect to four relations called the AS, IHX, multilinear and STU-like relations.

The surgery map $\psi\colon\A^{<,c}_d(H_\Z)\to Y_d\I\C/Y_{d+1}$ is defined as follows.
Let $J\in\A^{<,c}(H_\Z)$ be a connected Jacobi diagram whose labels of the univalent vertices are primitive elements in $H_\Z$.
We construct a graph clasper $G(J)$ from $J$ below,
and set $\psi(J)=(\Sigma_{g,1}\times[-1,1])_{G(J)}$.
First, replace the vertices of $J$ to disks,
and thicken the edges of $J$ to bands so that they respect the cyclic order of edges on each trivalent vertex.
Then, we obtain a compact oriented surface $S$ with orientation induced by the cyclic orders on the nodes.
Cutting a smaller disk in the interior of each disk coming from a univalent vertex,
we obtain the oriented surface $G(J)$.
We also endow the core of the resulting annulus, called a leaf, with the orientation induced by $\partial S$.
Next, we embed it in the trivial homology cylinder $\Sigma_{g,1}\times[-1,1]$ so that the core of each leaf represents the same homology class as the label of the univalent vertex. 
Moreover, we assume that each leaf lies in a horizontal slice of $\Sigma_{g,1}\times[-1,1]$
and that the vertical heights of leaves along $[-1,1]$ respect the total ordering of the univalent vertices of $J$.

Let $O(a_1,a_2,\ldots,a_d)\in \A_d^{<,c}(H_\Z)$ denote the Jacobi diagram appearing in Section~\ref{sec:Introduction} endowed with an arbitrary order in the univalent vertices.
Noting that $\tilde{\alpha}_d$ is the composition map
\[
Y_d\I\C/Y_{d+1}\xrightarrow{\alpha} \Ker\left(K_1(\widehat{\Q\pi}/\hat{I}^{d+1}) \to K_1(\widehat{\Q\pi}/\hat{I}^d)\right) \xrightarrow{\ldet_d} (H^{\otimes d})_{\Z_d},
\]
we have the following by Theorem~\ref{thm:value-1-loop} and Corollary~\ref{cor:gr}.

\begin{corollary}\label{cor:1-loopvalue}
\[
\tilde{\alpha}_d(\psi(O(x_1,x_2,\ldots,x_d)))=-x_1\otimes x_2\otimes\cdots \otimes x_d -(-1)^d x_d\otimes \cdots\otimes x_2\otimes x_1\in (H^{\otimes d})_{\Z_d}.
\]
\end{corollary}

%%%%%%%%%%%%%%%%%%%%%%%%%%
\section{Proof of Theorem~\ref{thm:main}}
\label{section:proof-of-main}
In this section, we prove Theorem~\ref{thm:main}.
For $M=(M, i)\in \C$, 
let us denote the mirror image as $\overline{M}=(\overline{M}, \bar{i})$,
where $\overline{M}$ is the 3-manifold obtained by reversing the orientation of $M$
and $\bar{i}\colon\partial(\Sigma_{g,1}\times [-1,1])\to \partial \overline{M}$ is defined by $\bar{i}(x,t)=i(x,-t)$.
Since taking the mirror images does not affect Torelli surgeries,
we have the following.
\begin{lemma}\label{lem:Y-equiv-mirror}
Let $d\ge1$.
If $M$ is $Y_d$-equivalent to $N$,
the mirror image $\overline{M}$ is $Y_d$-equivalent to $\overline{N}$.
\end{lemma}
In particular,
if $M\in Y_d\I\C$, the mirror image $\overline{M}$ also lies in $Y_d\I\C$.
Here, recall that $p_+\colon (H^{\otimes d})_{\Z_d}\to \A_{d,1}^c(H)$ is the homomorphism defined by
\[
p_+(x_1\otimes x_2\otimes\cdots\otimes x_d)
=O(x_1,x_2,\ldots,x_d),
\]
and $p_-\colon (H^{\otimes d})_{\Z_d} \to (H^{\otimes d})_{\Z_d}^-$ is the projection given by
\[
p_-(x_1\otimes x_2\otimes\cdots\otimes x_d)
=\frac{1}{2}(x_1\otimes x_2\otimes\cdots\otimes x_d-(-1)^dx_d\otimes \cdots \otimes x_2\otimes x_1).
\]
Theorem~\ref{thm:main} follows from the lemmas below.
\begin{lemma}\label{lem:-1-eigenspace}
For $M\in Y_d\I\C/Y_{d+1}$,
\begin{align*}
p_+(-\tilde{\alpha}_d(M)+\tilde{\alpha}_d(\overline{M}))&=0\in \A^c_{d,1}(H),\\
p_-(-\tilde{\alpha}_d(M)+\tilde{\alpha}_d(\overline{M}))&=(\Tr_d\circ\tau_d)(M)\in (H^{\otimes d})_{\Z_d}.
\end{align*}
\end{lemma}

\begin{lemma}\label{lem:1-eigenspace}
For $M\in Y_d\I\C/Y_{d+1}$,
\begin{align*}
(p_+\circ \tilde{\alpha}_d)(M\circ \overline{M})&=-2Z_{d,1}(M\circ \overline{M})\in \A^c_{d,1}(H),\\
(p_-\circ \tilde{\alpha}_d)(M\circ \overline{M})&=0\in (H^{\otimes d})_{\Z_d}.
\end{align*}
\end{lemma}

First, we prepare to prove Lemma~\ref{lem:-1-eigenspace}.
Let us briefly review the Magnus representation $r\colon \H\to \GL(2g,\widehat{\Q\pi})$ on the homology cobordism group $\H$ with coefficients in $\widehat{\Q\pi}$ following \cite{MaSa20}.
Let $(M,i)\in\C$.
By Lemma~\ref{lem:acylic1} and the homology long exact sequence of the triples $(M,i_{\pm}(\Sigma_{g,1}),*)$,
the homomorphisms
\[
(i_\pm)_*\colon H_1(\Sigma_{g,1},*;i_\pm^*\widehat{\Q\pi_1M})\to H_1(M,*;\widehat{\Q\pi_1M})
\]
are right $\widehat{\Q\pi_1M}$-module isomorphisms,
where $i_\pm^*\widehat{\Q\pi_1M}$ denote the pullbacks of the local system $\widehat{\Q\pi_1M}$ on $M$ under $i_\pm\colon \Sigma_{g,1}\hookrightarrow M$.
Let us denote
\[
r(M)=(i_-)_*^{-1}\circ (i_+)_*\colon H_1(\Sigma_{g,1},*;i_+^*\widehat{\Q\pi_1M})\to H_1(\Sigma_{g,1},*;i_-^*\widehat{\Q\pi_1M}).
\]
Using the spine $S=\bigcup_{i=1}^{2g}\gamma_i$ in Figure~\ref{fig:spine_gamma},
we have identifications
\begin{equation}
\label{eq:isom}
H_1(\Sigma_{g,1},*;i_{\pm}^*\widehat{\Q\pi_1M})\cong C_1(S;i_\pm^*\widehat{\Q\pi_1M})\cong (\widehat{\Q\pi_1M})^{2g}.
\end{equation}
Thus, $r(M)$ gives an element in $\GL(2g,\widehat{\Q\pi_1M})$.
Under the identification $i_-\colon \widehat{\Q\pi}\cong \widehat{\Q\pi_1M}$,
we may regard $r(M)$ as an element in $\GL(2g,\widehat{\Q\pi})$.
The map $r\colon \C\to \GL(2g,\widehat{\Q\pi})$ is a crossed homomorphism,
and is known to factor through the homology cobordism group $\H$.
See \cite[Section~3]{Sak08}, for example.

On the other hand,
consider the composition map
\[
H_1(\Sigma_{g,1},*;\widehat{\Q\pi})\xrightarrow{\partial_*} H_0(*;\widehat{\Q\pi})\cong\widehat{\Q\pi},
\]
where $\partial_*$ is the connecting  homomorphism of the homology long exact sequence of the pair $(\Sigma_{g,1},*)$.
It sends $\tilde{\gamma}_i\otimes u\mapsto (\gamma_i^{-1}-1)u$,
where $u\in \widehat{\Q\pi}$,
$\{\gamma_i\}_{i=1}^{2g}$ is the standard basis of $\pi_1\Sigma_{g,1}$,
and $\tilde{\gamma}_i$ is the lift of $\gamma_i$ starting at the basepoint of the universal covering.
Every element $v\in \hat{I}$ uniquely decomposes into the form $v=\sum_{i=1}^{2g}(\gamma_i-1)\partial_i v$ for some $\partial_iv\in\widehat{\Q\pi}$ as in \cite[Proposition~5.3]{MaSa20}.
Thus, $\partial_*$ induces an isomorphism $\Psi\colon H_1(\Sigma_{g,1},*;\widehat{\Q\pi})\cong \hat{I}$
whose inverse map $\Psi^{-1}\colon \hat{I}\to H_1(\Sigma_{g,1},*;\widehat{\Q\pi})$ is written as
\begin{equation}
\label{eq:psi}
\Psi^{-1}(v)= -\sum_{i=1}^{2g}\tilde{\gamma}_i\otimes \gamma_i\partial_i v \in H_1(\Sigma_{g,1},*;\widehat{\Q\pi}).
\end{equation}
In the same way as the proof of \cite[Theorem~7.2]{MaSa20},
we have:

\begin{lemma}
\label{lem:T_1}
The diagram
\[
\xymatrix{
(\widehat{\Q\pi})^{2g}\ar[d]_-{r(M)\circ \sigma_M}
&H_1(\Sigma_{g,1},*;\widehat{\Q\pi})\ar[l]_-{\cong}\ar[r]^-{\Psi}
&\hat{I}\ar[d]^-{\sigma_M}\\
(\widehat{\Q\pi})^{2g}
&H_1(\Sigma_{g,1},*;\widehat{\Q\pi})\ar[l]^-{\cong}\ar[r]_-{\Psi}
&\hat{I}
}
\]
commutes for $M\in\C$,
where the isomorphism $H_1(\Sigma_{g,1},*;\widehat{\Q\pi})\cong (\widehat{\Q\pi})^{2g}$ is defined in the same way as \textup{\eqref{eq:isom}},
and $\sigma_M=(i_-)_*^{-1}\circ (i_+)_*\in \Aut\widehat{\Q\pi}$.
\end{lemma}

Massuyeau and Sakasai~\cite{MaSa20} related the Magnus representation over $\H$ to Morita's trace map and some variants.
The Enomoto-Satoh trace is considered as a refinement of Morita's trace map in \cite{Mor93}.
In the following,
we relate the Magnus representation $r\colon\H\to\GL(2g,\widehat{\Q\pi})$ to the Enomoto-Satoh trace.
Let $\Aut\hat{T}$ and $\Aut\widehat{\Q\pi}$ denote the group of filtration-preserving automorphisms of $\hat{T}$ and $\widehat{\Q\pi}$, respectively.
We also denote by $\IAut\hat{T}$ and $\IAut\widehat{\Q\pi}$ the subgroups of the automorphisms which induce the identity at the graded levels.
In \cite{Sat12}
the degree-preserving $\GL(H)$-equivariant $\Q$-linear map
\[
\Tr\colon \Der (\hat{T},\hat{T}_2)\cong \Hom(H,\hat{T}_2)\to H_1^{\Lie}(\hat{T}_1)\cong \prod_{d=1}^\infty (H^{\otimes d})_{\Z_d}
\]
was constructed,
whose degree $d$ part is described by the composition map
\[
\Tr_d\colon \Hom(H,H^{\otimes(d+1)})\cong H^*\otimes H^{\otimes(d+1)}\xrightarrow{C} H^{\otimes d} \twoheadrightarrow (H^{\otimes d})_{\Z_d}
\]
in Section~\ref{sec:Introduction}.
See also \cite{EnSa14}.
Let us define $\Mag\colon \Aut\widehat{\Q\pi}\to \GL(2g,\widehat{\Q\pi})$ by
\[
\Mag(\varphi)\colon (\widehat{\Q\pi})^{2g}\xrightarrow{\varphi^{-1}}(\widehat{\Q\pi})^{2g}\xrightarrow{\Psi}\hat{I}\xrightarrow{\varphi} \hat{I}\xrightarrow{\Psi^{-1}}(\widehat{\Q\pi})^{2g},
\]
Note that $\Mag(\sigma_M)=r(M)$ for $M\in\H$ by Lemma~\ref{lem:T_1}
and that $\Mag$ is a crossed homomorphism.
Consider the two logarithmic maps
\begin{align*}
\log\colon& \IAut\hat{T}\to \Der(\hat{T},\hat{T}_2),\\
\log\colon& \Ker\left(\GL(2g,\widehat{\Q\pi})\to \GL(2g,\Q)\right) \to 
M(2g,\hat{I}),
\end{align*}
where the former is defined by
$\log(\varphi)=\sum_{k=1}^\infty\frac{(-1)^{k-1}}{k}(\varphi-\id)^k$
and the latter is explained in \eqref{eq:log}.
It is well-known that the image of the former logarithmic map is in $\Der(\hat{T},\hat{T}_2)$.
The proof is almost the same as that of \cite[Proposition~5.12]{Mas12}.

Note that as in \cite[Section~6.1]{MaSa20},
for $M\in J_d\C$,
the element $(\log\circ \theta_*)(\sigma_M)$ in $\Der(\hat{T},\hat{T}_{d+1}/\hat{T}_{d+2})$ coincides with the $d$th Johnson homomorphism $\tau_d(M)\in\Hom(H,H^{\otimes(d+1)})$
under the natural identification.
The diagram
\[
\xymatrix{
\IAut\widehat{\Q\pi}\ar[r]^-{\theta_*}_-{\cong}\ar[rd]_-{\Mag}&\IAut\hat{T} \ar[r]^-{\log}&\Der(\hat{T},\hat{T}_2)\ar[r]^{\Tr}&H_1^{\Lie}(\hat{T}_1)\\
&\Ker(\GL(2g,\widehat{\Q\pi})\to\GL(2g,\Q))\ar[r]_-{\log}&M(2g,\hat{I})\ar[r]_-{\tr}&H_1^\Lie(\hat{I})\ar[u]^-{\cong}_-{\theta_*}
}
\]
describes some relation between the Enomoto-Satoh trace and the Magnus representation,
where $\theta_*(\varphi)=\theta\circ\varphi\circ\theta^{-1}$ for $\varphi\in\IAut\widehat{\Q\pi}$.
The diagram does not commute in general.
However, in \cite[Lemma~4.3]{MaSa20},
by taking some quotient of $\hat{T}$,
Massuyeau and Sakasai treated a commutative diagram similar to the above.
Here, we prove that the degree $d$ part of the above diagram commutes if we restrict to $\Ker(\Aut\widehat{\Q\pi}\to \Aut(\widehat{\Q\pi}/\hat{I}^{d+1}))$.
Let us denote by $\tr_d\colon M(2g,\hat{I}^d)\to \hat{I}^d/\hat{I}^{d+1}$ the trace map,
and denote by the same symbol $\theta_*$ the natural homomorphism $\hat{I}^d/\hat{I}^{d+1}\to (H^{\otimes d})_{\Z_d}$ induced by $\theta$.
Note that $\tr_d$ is different from $\Tr_d$ defined above.

\begin{lemma}
\label{lem:degree-d}
For $\varphi\in \Ker(\Aut\widehat{\Q\pi}\to \Aut(\widehat{\Q\pi}/\hat{I}^{d+1}))$,
we have
\[
(\Tr_d\circ\log\circ \theta_*)(\varphi)
=(\theta_*\circ \tr_d\circ\log\circ\Mag)(\varphi)\in (H^{\otimes d})_{\Z_d}.
\]
\end{lemma}

\begin{proof}
By definition,
$\Mag(\varphi)\in \Ker(\GL(2g,\widehat{\Q\pi})\to \GL(2g,\widehat{\Q\pi}/\hat{I}^d))$.
Thus, we have
\[
\log(\Mag(\varphi))=\Mag(\varphi)-I_{2g}\in M(2g,\hat{I}^d/\hat{I}^{d+1}).
\]
Note that we identify $H_1(\Sigma_{g,1},*;\widehat{\Q\pi})$ and $(\widehat{\Q\pi})^{2g}$ by the right $\widehat{\Q\pi}$-module isomorphism which sends $\tilde{\gamma}_i\otimes 1$ to the $i$th standard unit vector for $1\le i\le 2g$.
By the description of $\Psi^{-1}$ in (\ref{eq:psi}),
we obtain
\begin{align*}
\tr_d(\Mag(\varphi)-I_{2g})
&=\tr_d(\Psi^{-1}\circ\varphi\circ\Psi\circ\varphi^{-1}-I_{2g}) \\
&=-\sum_{i=1}^{2g}\gamma_i\partial_i(((\varphi-\id_{\hat{I}})\circ\Psi)(\tilde{\gamma}_i\otimes 1))\\
&=-\sum_{i=1}^{2g}\gamma_i\partial_i(\varphi(\gamma_i)^{-1}-\gamma_i^{-1})\\
&=\sum_{i=1}^{2g}\partial_i(\varphi(\gamma_i)-\gamma_i)\in \hat{I}^d/\hat{I}^{d+1}.
\end{align*}
Here, the last equality follows from the fact that
\[
\varphi(\gamma_i)^{-1}-\gamma_i^{-1}
=-\varphi(\gamma_i)^{-1}(\varphi(\gamma_i)-\gamma_i)\gamma_i^{-1}
=-(\varphi(\gamma_i)-\gamma_i)\in \hat{I}^{d+1}/\hat{I}^{d+2}.
\]

On the other hand,
as an element of $\Der(\hat{T},\hat{T}_2/\hat{T}_{d+1})=\Hom(H,\hat{T}_2/\hat{T}_{d+1})$,
\[
(\log\circ\theta_*)(\varphi)=
\theta\circ\log\varphi\circ\theta^{-1}
=\theta\circ\varphi\circ\theta^{-1}-\id_H.
\]
We denote by $[x]\in H_\Z$ the homology class represented by $x\in\pi$.
Every element $X\in \hat{T}_1$ uniquely decomposes into the form $X=\sum_{i=1}^{2g}[\gamma_i]\partial_i X$ for some $\partial_i X\in\hat{T}$.
If we denote $\theta^{-1}([\gamma_i])=(\gamma_i-1)+v_i$ for some $v_i\in \hat{I}^2$,
we have
\begin{align*}
\Tr_d(\theta\circ\varphi\circ\theta^{-1}-\id_H)
&=\sum_{i=1}^{2g}\partial_i(\theta\circ\varphi\circ\theta^{-1}([\gamma_i])-[\gamma_i])\\
&=\sum_{i=1}^{2g}\partial_i(\theta(\varphi(\gamma_i+v_i)-(\gamma_i+v_i))).
\end{align*}
Since 
\[
\varphi(xy)-xy=(\varphi(x)-x)\varphi(y)+x(\varphi(y)-y)\in \hat{I}^{d+2}
\]
for $x,y\in \hat{I}$, 
we see that it is equal to
\begin{align*}
\sum_{i=1}^{2g}\partial_i(\theta(\varphi(\gamma_i)-\gamma_i))
=\sum_{i=1}^{2g}\theta_*(\partial_i(\varphi(\gamma_i)-\gamma_i))\in (H^{\otimes d})_{\Z_d}.
\end{align*}
Thus, we see that the degree $d$ part of the diagram commutes.
\end{proof}

The Magnus representation is related to our torsion as follows.
\begin{theorem}
\label{thm:torsion-magnus}
For $M\in\C$,
\[
\tilde{\alpha}(M)^{-1}\cdot (\sigma_M)_* \tilde{\alpha}(\overline{M})=r(M) \in K_1(\widehat{\Q\pi})/\tilde{\rho}(H_\Z).
\]
Moreover, for $M\in\I\C$,
\[
\tilde{\alpha}(M)^{-1}\cdot (\sigma_M)_* \tilde{\alpha}(\overline{M})=r(M) \in K_1(\widehat{\Q\pi}).
\]
\end{theorem}

\begin{proof}
The proof is almost the same as \cite[Theorem~5.3(1)]{Sak15}.
Consider the square matrix 
$A(P)=\begin{pmatrix}
A\\
B
\end{pmatrix}$ of size $2g+l$ for some balanced presentation of $\pi_1M$ as in Section~\ref{section:compute-RT}.
Let $C$ be a $2g\times (2g+l)$-matrix defined by
$C=
\begin{pmatrix}
\overline{\frac{\partial r_j}{\partial i_-(\gamma_i)}}
\end{pmatrix}_{\substack{1\le i\le 2g\\1\le j\le 2g+l}}$.
As in \cite[Proposition~3.9(2)]{Sak08},
we have
\[
r(M)=-\tilde{\rho}\left(C\begin{pmatrix}
A\\B
\end{pmatrix}^{-1}
\begin{pmatrix}
I_{2g}\\ O_{l,2g}
\end{pmatrix}\right)
\in \GL(2g,\widehat{\Q\pi}).
\]
Recall that the determinant map $\det\colon K_1(\widehat{\Q\pi})\to (\widehat{\Q\pi}^{\times})_{\ab}$ is an isomorphism as in \cite[Corollary~2.2.6]{Ros94}.
Thus, as in the proof of \cite[Theorem~5.3(1)]{Sak15},
we obtain
\[
r(M)\cdot \tilde{\rho}\begin{pmatrix}
A\\B
\end{pmatrix}
=\tilde{\rho}
\begin{pmatrix}
C\\
B
\end{pmatrix}
=(\sigma_M)_*\tilde{\alpha}(\overline{M})\in K_1(\widehat{\Q\pi})/\tilde{\rho}(H_\Z).
\]
It implies that
\[
\tilde{\alpha}(M)^{-1}\cdot (\sigma_M)_* \tilde{\alpha}(\overline{M})=r(M) \in K_1(\widehat{\Q\pi})/\tilde{\rho}(H_\Z)
\]
for $M\in \C$.

Next, let $M\in\I\C$.
Since $\ldet_1\colon \Ker\left(K_1(\widehat{\Q\pi}/\hat{I}^2)\to K_1(\widehat{\Q\pi}/\hat{I})\right)\to H$ is an isomorphism,
it suffices to show that
\[
\ldet_1(\tilde{\alpha}(M)^{-1}\cdot (\sigma_M)_* \tilde{\alpha}(\overline{M}))
=\ldet_1(r(M)) \in H.
\]
Since $\sigma_M$ acts on $K_1(\widehat{\Q\pi}/\hat{I}^2)$ trivially,
by Lemma~\ref{lem:euler-str}, we have
\begin{align*}
\ldet_1(\tilde{\alpha}(M)^{-1}\cdot (\sigma_M)_* \tilde{\alpha}(\overline{M}))
&=-\ldet_1(\tilde{\alpha}(M))+\ldet_1(\tilde{\alpha}(\overline{M})) \\
&=(C\circ\tau_1)(M).
\end{align*}
On the other hand,
noting that $\epsilon(r(M))\in\GL(2g,\Q)$ is the identity matrix,
by Lemma~\ref{lem:degree-d},
we have
\[
\ldet_1(r(M))=(\theta_*\circ \tr_1\circ \log)(r(M))
=(\Tr_1\circ\log\circ\theta_*)(\sigma_M)
=(C\circ\tau_1)(M).
\]
Thus, we obtain the equality.
\end{proof}

\begin{proof}[Proof of Lemma~\ref{lem:-1-eigenspace}]
Recall the isomorphism 
\[
\ldet_d\colon \Ker\left(K_1(\widehat{\Q\pi}/\hat{I}^{d+1}) \to K_1(\widehat{\Q\pi}/\hat{I}^{d})\right) \to (H^{\otimes d})_{\Z_d}
\]
defined in Corollary~\ref{cor:gr}.
As we can see from the isomorphism,
$\sigma_M$ acts on $\Ker(K_1(\widehat{\Q\pi}/\hat{I}^{d+1}) \to K_1(\widehat{\Q\pi}/\hat{I}^{d}))$ trivially.
By Theorem~\ref{thm:torsion-magnus} and Lemma~\ref{lem:degree-d},
we have
\begin{align*}
-\tilde{\alpha}_d(M)+\tilde{\alpha}_d(\overline{M})
= \ldet_d(\tilde{\alpha}(M)^{-1}\cdot (\sigma_M)_*^{-1}\tilde{\alpha}(\overline{M}))
&= (\ldet_d\circ r)(M)\\
&= (\theta_*\circ \tr_d\circ \log \circ r)(M)\\
&= (\Tr_d\circ\log\circ\theta_*)(\sigma_M)\\
&= (\Tr_d\circ\tau_d)(M)
\end{align*}
for $M\in Y_d\I\C/Y_{d+1}$.
Since the Enomoto-Satoh trace takes value in $(H^{\otimes d})_{\Z_d}^{-}$ as shown in \cite[Theorem~4.2(1)]{Con15},
we see that
\[
p_-(-\tilde{\alpha}_d(M)+\tilde{\alpha}_d(\overline{M}))=(\Tr_d\circ\tau_d)(M),\quad
p_+(-\tilde{\alpha}_d(M)+\tilde{\alpha}_d(\overline{M}))=0.
\]
\end{proof}

Next, we prepare to prove Lemma~\ref{lem:1-eigenspace}.
Let $G$ be a graph clasper in a 3-manifold $M$,
which is an embedded compact connected surface consisting of 3 kinds of constituents called leaves, nodes, and bands.
The clasper surgery is surgery along some framed link in $M$ associated with the graph clasper $G$,
and there are two conventions for it.
The first framed link is given in \cite[Section~2]{Hab00C},
which is the same as that used in \cite[Section~7 and Appendix~E]{Oht02}, \cite{MaMe13}, and \cite{HaMa12},
and let us denote it by $L(G)$.
Let us denote by $L'(G)$ the framed link in \cite[Figure~4]{GGP01},
which is also used in \cite{Gou99} and \cite{Gus00}.
We also denote by $M_L$ the 3-manifold obtained by surgery along a link $L$ in $M$.

\begin{figure}[h]
 \centering
 \includegraphics[width=0.6\textwidth]{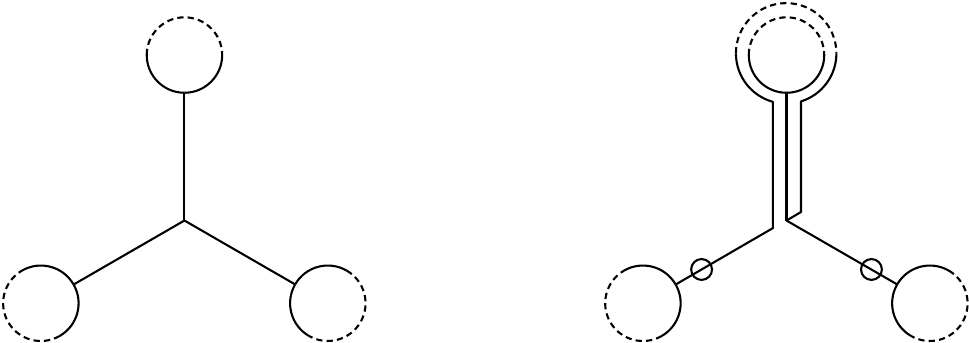}
 \caption{Claspers $Y_1$ (left) and $Y_2$ (right), where $\ominus$ represents a negative half-twist.}
 \label{fig:GHY}
\end{figure}

\begin{lemma}
\label{lem:GH-surgeries0}
Let $Y_1$ and $Y_2$ be tree claspers of degree $1$ in a $3$-manifold $M$ as in Figure~\ref{fig:GHY}.
Then, $M_{L'(Y_1)}$ is homeomorphic to $M_{L(Y_2)}$.
\end{lemma}

\begin{proof}
Let us transform the framed link $L'(Y_1)$ as follows:
\begin{center}
 \includegraphics[width=0.99\textwidth]{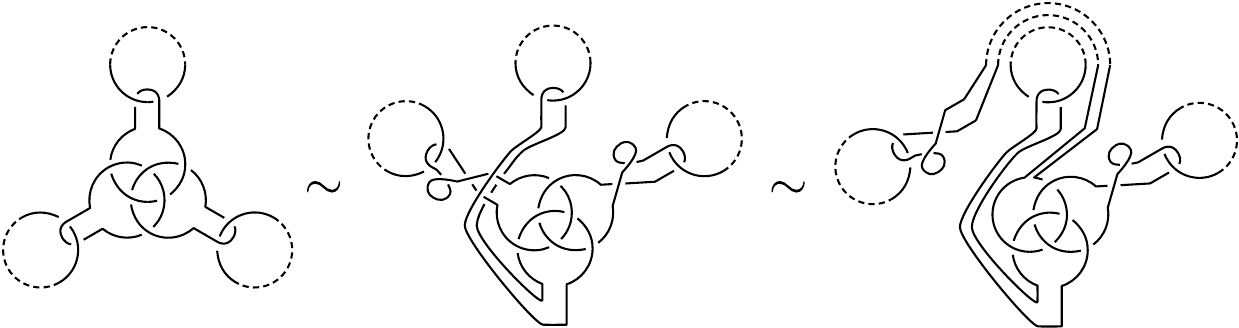}
\end{center}
Here the first $\sim$ is obtained by isotopy and the second one follows from a Kirby move.
The resulting framed link is $L(Y_2)$.
\end{proof}

\begin{figure}[h]
 \centering
 \includegraphics[width=0.9\textwidth]{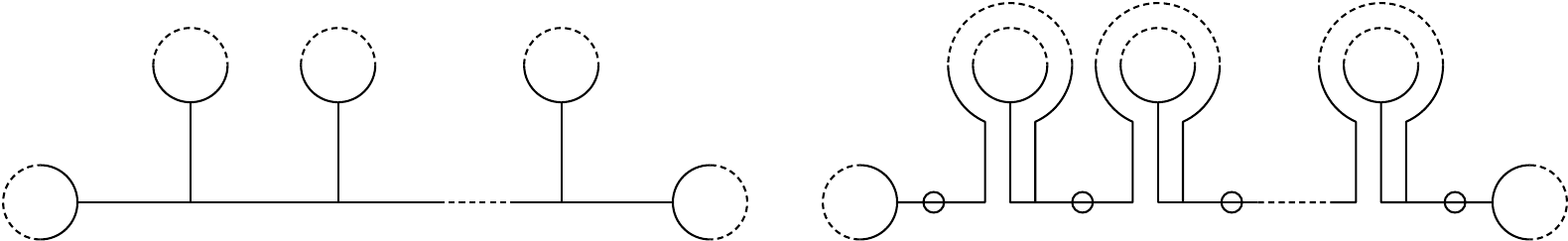}
 \caption{Claspers $G_1$ (left) and $G_2$ (right).}
 \label{fig:GHT}
\end{figure}

\begin{lemma}
\label{lem:GH-surgeries1}
Let $G_1$ and $G_2$ be tree claspers of degree $d$ in a $3$-manifold $M$ as in Figure~\ref{fig:GHT}.
Then, $M_{L'(G_1)}$ is homeomorphic to $M_{L(G_2)}$.
\end{lemma}

\begin{proof}
We give the proof for $d=2$.
The cases $d\geq 3$ are shown in much the same way.
By \cite[Theorem~2.4]{GGP01}, $M_{L'(G_1)}$ is homeomorphic to $M_{L'(G_3)}$, where $G_3$ is the following clasper:
\begin{center}
 \includegraphics[width=0.4\textwidth]{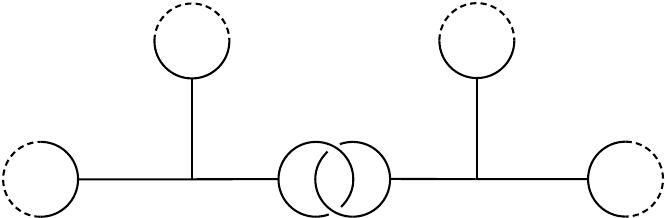}
\end{center}
It follows from Lemma~\ref{lem:GH-surgeries0} that $M_{L'(G_3)}$ is homeomorphic $M_{L(G_4)}$, where $G_4$ is the clasper on the left:
\begin{center}
 \includegraphics[width=0.9\textwidth]{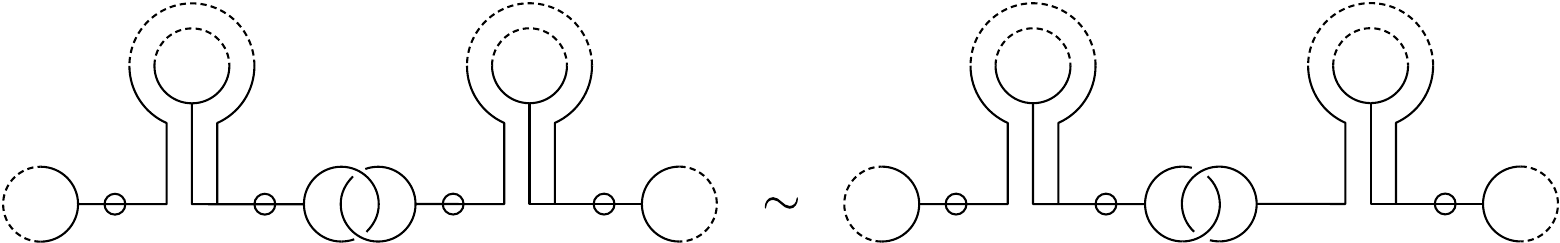}
\end{center}
Here the right one is obtained by the flip of a leaf.
Finally, \cite[Move~2]{Hab00C} completes the proof.
\end{proof}

\begin{lemma}
\label{lem:GHsurgeries2}
Let $G$ be a graph clasper of degree $d$ embedded in $M$.
\begin{enumerate}
\item When $b_1(G)$ is odd,
$M_{L(G)}\sim_{Y_{d+1}}M_{L'(G)}$.

\item When $b_1(G)$ is even,
$M_{L(G')}\sim_{Y_{d+1}}M_{L'(G)}$,
where $G'$ is a graph clasper obtained by applying a half-twist to one edge in $G$.
\end{enumerate}
\end{lemma}

\begin{proof}
First, consider the case when $G$ is a tree clasper, namely, $b_1(G)=0$.
By the IHX relation (see \cite[Lemma~E.22]{Oht02} and \cite[Theorem~4.11]{GGP01}, for example),
we may assume that the clasper $G$ is homeomorphic to $G_1$ in Figure~\ref{fig:GHT} as surfaces.
By Lemma~\ref{lem:GH-surgeries1},
$M_{L'(G)}$ is homeomorphic to $M_{L(G_2)}$,
and by the AS relation (see \cite[Lemmas~E.7 and E.9]{Oht02} and \cite[Lemma~4.4 and Corollary~4.6]{GGP01}, for example),
$M_{L(G_2)}$ is $Y_{d+1}$-equivalent to $M_{L(G')}$.

Next, consider the case when $b_1(G)>0$. 
Applying Move~2 in \cite[Figure~9]{Hab00C} to several edges of $G$,
we can replace $G$ by a tree clasper $G_1$ such that $L(G_1)=L(G)$.
In the same way, 
applying \cite[Theorem~2.4]{GGP01} to the same edges,
we can also replace $G$ by another tree clasper $G_2$ such that $L'(G_2)=L'(G)$.
The claspers $G_1$ and $G_2$ differ only in half-twists along the edges to which Move~2 in \cite[Figure~9]{Hab00C} or \cite[Theorem~2.4]{GGP01} was applied.
Thus, we have 
$M_{L(G_1)}\sim_{Y_{d+1}} M_{L(G'_2)}$ when $b_1(G)$ is odd,
and $M_{L(G_1)}\sim_{Y_{d+1}} M_{L(G_2)}$ when $b_1(G)$ is even,
where $G'_2$ is a graph clasper obtained by applying a half-twist to one edge in $G_2$.
since $b_1(G_2)=0$, we also have $M_{L(G'_2)}\sim_{Y_{d+1}} M_{L'(G_2)}$.
In summary, when $b_1(G)$ is odd,
we have
\begin{align*}
M_{L(G)}
=M_{L(G_1)}
\sim_{Y_{d+1}} M_{L(G'_2)}
\sim_{Y_{d+1}} M_{L'(G_2)}
=M_{L'(G)}.
\end{align*}
The case when $b_1(G)$ is even is proved in the same way.
\end{proof}

For a connected Jacobi diagram $J\in \A_d^{<,c}(H_\Z)$,
let $\overline{J}\in \A_d^{<,c}(H_\Z)$ denote the Jacobi diagram obtained by reversing the order among the vertices in $J$.
Since the AS, IHX, multilinear, and STU-like relations are preserved,
the involution $\rr\colon \A^{<,c}(H_\Z)\to\A^{<,c}(H_\Z)$ defined by $\rr(J)=(-1)^{b_1(J)+1}\bar{J}$ is well-defined.
Here, we use the same symbol as $\rr\colon (H^{\otimes d})_{\Z_d}\to (H^{\otimes d})_{\Z_d}$ in Section~\ref{sec:Introduction}.

\begin{proposition}
For $J\in \A_d^{<,c}(H_\Z)$,
\[
\overline{\psi(J)}=\psi(\rr(J))\in Y_d\I\C/Y_{d+1}.
\]
\end{proposition}

\begin{proof}
Let $G(J)$ denote the graph clasper appeared right before Corollary~\ref{cor:1-loopvalue} such that $\psi(J)=(\Sigma_{g,1}\times[-1,1])_{G(J)}$.
If we take the mirror image,
the vertical heights of the leaves of $G(J)$ are reversed,
and the surface $G(J)$ is sent to $G(\overline{J})$ in $\Sigma_{g,1}\times [-1,1]$.
Moreover, clasper surgery in Habiro's sense along $G(J)$ turns into clasper surgery along $G(\overline{J})$ in Goussarov's sense.
Thus, by Lemma~\ref{lem:GHsurgeries2},
we have equalities
\begin{align*}
\overline{\psi(J)}
=\overline{(\Sigma_{g,1}\times[-1,1])_{L(G(J))}}
&=(\Sigma_{g,1}\times[-1,1])_{L'(G(\overline{J}))}\\
&=(-1)^{b_1(J)+1}\psi(\overline{J})\\
&=\psi(\rr(J))
\end{align*}
in $Y_d\I\C/Y_{d+1}$.
\end{proof}

\begin{proof}[Proof of Lemma~\ref{lem:1-eigenspace}]
We may assume $M=\psi\left(\sum_{b=0}^\infty J_b\right) \in Y_d\I\C/Y_{d+1}$,
where $J_b$ is a sum of connected Jacobi diagrams whose first Betti numbers are $b$.
Let $\col(v)\in H_\Z$ denote the label of a uni-trivalent vertex $v$ in a Jacobi diagram.
By the STU-like relation,
\begin{align*}
\rr(J_1)&=J_1+(\text{Jacobi diagrams with $b_1\ge2$}),\\
\rr(J_0)&=-J_0+C(J_0)+(\text{Jacobi diagrams with $b_1\ge2$}),
\end{align*}
where $C(J_0)$ denotes the sum of all ways of gluing two univalent vertices $v$ and $w$ in $J_0$ such that $v<w$ with coefficient $\col(v)\cdot \col(w) \in \Z$.
Thus, we obtain
\[
\tilde{\alpha}_d(M)+\tilde{\alpha}_d(\overline{M})
=\tilde{\alpha}_d(\psi(J_0+J_1+r(J_0)+r(J_1)))
=\tilde{\alpha}_d(\psi(C(J_0)+2J_1)).
\]
by Theorem~\ref{thm:k-loop}.
Let $\A_{d,\ge 1}^{<,c}(H_\Z)$ denote the submodule of $\A_d^{<,c}(H_\Z)$ generated by $k$-loop Jacobi diagrams for $k\ge 1$.
By the IHX-relation,
we see that the module $\A_{d,\ge1}^{<,c}(H_\Z)$ is generated by 1-loop Jacobi diagrams of the form $O(a_1,a_2,\ldots,a_d)$ and $k$-loop Jacobi diagrams for $k\ge2$ with arbitrary orders in the same way as \cite[Proposition~5.1]{NSS22GT}.
Thus, by Theorem~\ref{thm:k-loop} and Corollary~\ref{cor:1-loopvalue},
we see that
\[
(p_-\circ \tilde{\alpha}_d)(M\circ \overline{M})=0\in (H^{\otimes d})_{\Z_d}.
\]
For a $1$-loop Jacobi diagram $J\in \A^{<,c}_d(H_\Z)$,
let $\chi'_{d,1}(J)\in \A^c_{d,1}(H) $ denote the same Jacobi diagram obtained by forgetting the order,
and for a $k$-loop Jacobi diagram $J\in \A^{<,c}_d(H_\Z)$ where $k\ge2$,
set $\chi'_{d,1}(J)=0\in \A^c_{d,1}(H)$.
The induced homomorphism $\chi'_{d,1}\colon \A_{d,\ge1}^{<,c}(H_\Z)\to \A_{d,1}^c(H)$ is equal to the composition
\[
\A_{d,\ge1}^{<,c}(H_\Z)\xrightarrow{\incl} \A_d^{<,c}(H_\Z)\xrightarrow{\chi_d^{-1}}\A^c_d(H)\xrightarrow{\pr} \A^c_{d,1}(H),
\]
where $\chi_d$ is the restriction of $\chi$ in \cite[Section~3.1]{HaMa09} to the degree $d$ part.
By the definition of $p_+$ and Corollary~\ref{cor:1-loopvalue},
we have
\[
(p_+\circ\tilde{\alpha}_d)(M\circ\overline{M})=-2\chi'_{d,1}(C(J_0)+2J_1)\in\A^c_{d,1}(H).
\]
On the other hand,
by the commutative diagram in \cite[Theorem~6.7]{HaMa12} and \cite[Proposition~3.1]{HaMa09}, we have
\begin{align*}
Z_{d,1}(M)&=(Z_{d,1}\circ\psi)\left(\sum_{b=0}^\infty J_b\right)
=\chi_{d,1}' \left(\frac{1}{2}C(J_0)+J_1\right),\\
Z_{d,1}(\overline{M})&=\chi_{d,1}' \left(\frac{1}{2}C(J_0)+J_1\right).
\end{align*}
Thus, we obtain
\[
(p_+\circ \tilde{\alpha}_d)(M\circ \overline{M})
=-2Z_{d,1}(M\circ \overline{M})\in \A^c_{d,1}(H).
\]
\end{proof}

%\bibliographystyle{abbrv}
%\bibliography{RT-torsion}
\def\cprime{$'$} \def\cprime{$'$} \def\cprime{$'$}

\end{document}